\documentclass[11pt, twoside, leqno]{article}

\usepackage{natbib}
\usepackage[english]{babel}
\usepackage[latin1]{inputenc}
\usepackage{csquotes}
\usepackage{url}
\usepackage[breaklinks]{hyperref}
\hypersetup{colorlinks=true,linkcolor  = black, citecolor=blue, breaklinks=true}

\usepackage{fancyhdr}
\usepackage{indentfirst}
\usepackage{graphicx}
\usepackage{float}
\usepackage{amsmath, amsfonts, amssymb, amsthm}
\usepackage{geometry}
\usepackage{setspace}
\usepackage{stackrel}
 \usepackage[sc]{titlesec}

\usepackage{lipsum}

\newcommand\blfootnote[1]{%
  \begingroup
  \renewcommand\thefootnote{}\footnote{#1}%
  \addtocounter{footnote}{-1}%
  \endgroup
}

\numberwithin{equation}{section}  

\newcommand\numberthis{\addtocounter{equation}{1}\tag{\theequation}}

\theoremstyle{plain}
\newtheorem{thm}{Theorem}[section]
\newtheorem{thmy}{Theorem}
\newenvironment{thmx}{\stepcounter{thm}\begin{thmy}}{\end{thmy}}

\newtheorem{thmz}{Theorem}

\newtheorem{prop}[thm]{Proposition} 
\newtheorem{lem}[thm]{Lemma}

\makeatletter
\newtheorem*{rep@theorem}{\rep@title}
\newcommand{\newreptheorem}[2]{%
\newenvironment{rep#1}[1]{%
 \def\rep@title{#2 \ref{##1}}%
 \begin{rep@theorem}}%
 {\end{rep@theorem}}}
\makeatother  

\newreptheorem{theorem}{Theorem}

\theoremstyle{definition}
\newtheorem{rk}[thm]{Remark}

\newtheorem{defin}[thm]{Definition}

\newcommand{\N}{\mathbb{N}} 
\newcommand{\R}{\mathbb{R}}

\renewcommand{\Im}{\operatorname{Im}} 
\renewcommand{\emptyset}{\varnothing} 
\newcommand{\e}{\varepsilon} 
\newcommand{\res}{\hbox{ {\vrule height .22cm}{\leaders\hrule\hskip.2cm} }} 
\newcommand{\Atop}[2]{\genfrac{}{}{0pt}{}{#1}{#2}}

\newcommand{\dist}{\operatorname{dist}}

\newcommand{\supp}{\operatorname{supp}}

\title{\textsc{Sufficient conditions for $C^{1,\alpha}$ parametrization and rectifiability}}
\author{Silvia Ghinassi}
\date{}

\makeatletter
\newcommand\Author{\textsc{Silvia Ghinassi}}
\let\Title\@title
\makeatother

\begin{document}
\maketitle

\begin{abstract}

We say a measure is $C^{1,\alpha}$ $d$-rectifiable if there is a countable union of $C^{1,\alpha}$ $d$-surfaces whose complement has measure zero. We provide sufficient conditions for a Radon measure in $\R^n$ to be $C^{1,\alpha}$ $d$-rectifiable, with $\alpha \in (0,1]$. The conditions involve a Bishop-Jones type square function and all statements are quantitative in that the $C^{1,\alpha}$ constants depend on such a function. Along the way we also give sufficient conditions for $C^{1,\alpha}$ parametrizations for Reifenberg flat sets in terms of the same square function. Key tools for the proof come from David and Toro's Reifenberg parametrizations of sets with holes in the H\"{o}lder and Lipschitz categories. 
\end{abstract}
\blfootnote{\emph{Date:} November 28th, 2019}

\section{Introduction}

\subsection{Background}

Recall that a set $E$ in $\R^n$ is \emph{Lipschitz image} $d$-rectifiable -- countably $d$-rectifiable in Federer's terminology -- if there exist countably many Lipschitz maps $f_i \colon \R^d \to \R^n$ such that $\mathcal{H}^d( E \setminus \bigcup_i f_i(\R^d))=\mathcal{H}^d\res_E( \R^n \setminus \bigcup_i f_i(\R^d))=0$, where $\mathcal{H}^d$ denotes the $d$-dimensional Hausdorff measure.
In this paper, we investigate sets that can be covered by images of more regular maps (see Section \ref{results} for the statements of the main results and Section \ref{motivation} for motivations).

We say that a set $E$ in $\R^n$ is $C^{1,\alpha}$ $d$-rectifiable if there exist countably many continuously differentiable Lipschitz maps $f_i \colon \R^d \to \R^n$ with $\alpha$-H\"{o}lder derivatives such that
\begin{equation}
\mathcal{H}^d\res_E \left( \R^n \setminus \bigcup_i f_i(\R^d)\right)=0.
\end{equation}

For \emph{Lipschitz image} rectifiability, we could replace the class of Lipschitz images with bi-Lipschitz images, $C^1$ images, Lipschitz graphs, or $C^1$ graphs without changing the class of rectifiable sets; see Theorem 3.2.29 in   \cite{federer} and   \cite{davidbook} for proofs of these equivalences. From now on we will refer to \emph{Lipschitz image} rectifiability simply as rectifiability.

On the other hand, rectifiability of order $C^{1,\alpha}$ does not imply rectifiability of order $C^{1,\alpha'}$ for any $0 \leq \alpha < \alpha' \leq 1$. More generally, $C^{k-1,1}$ rectifiability is equivalent to $C^k$ rectifiability (Proposition 3.2 in \cite{anzellottiserapioni}), while there are $C^{k,s}$ rectifiable sets that are not $C^{m,t}$ rectifiable, whenever $k,m \geq 1$ and $k + s <m+t$ (Proposition 3.3 in \cite{anzellottiserapioni}). For completeness, we include the proofs of these results in the Appendix, as Propositions \ref{c1=c2} and \ref{c1snoc1t}.

Classical rectifiability of sets has been widely studied and characterized, see \cite{mattila} for an exposition. However, a quantitative theory of rectifiability was only developed in the late 1980s to study connections between rectifiable sets and boundedness of singular integral operators.
Peter Jones in \cite{atst} gives a quantitative control on the length of a rectifiable curve in terms of a sum of $\beta$ numbers. These numbers capture, at a given scale and location, how far a set is from being a line. Jones' proof was generalized to $1$-dimensional objects in $\R^n$ by Okikiolu in   \cite{okikiolu} and in Hilbert spaces by Schul in \cite {raananhilbert}. 


In  \cite{davidtoro} David and Toro prove that one-sided Reifenberg flat sets admit a bi-H\"{o}lder parametrization, which is a refinement of Reifenberg's original proof in \cite{reifen}. Moreover, if one also assumes square summability of the $\beta$'s the parametrization is actually bi-Lipschitz (see also \cite{toro1995}). To better understand this, consider a variation of the usual snowflake. Start with the unit segment $[0,1]$, and let this be step $0$. At each step $i$ we create an angle of $\alpha_i$ by adding to each segment of length $2^{-i+1}$ an isosceles triangle in the center, with base $\frac{2^{-i+1}}{3}$ and height $\frac{2^{-i+1}\alpha_i}{6}$ (since the $\alpha_i$'s are small we can use a first order approximation). Then the resulting curve is rectifiable (i.e. has finite length) if and only if $\sum_i\alpha_i^2 < \infty$ (see Exercise 10.16 in   \cite{bishopperes}).

Consider now a smoothened version of the snowflake where we stop after a finite number of iterations. This set is clearly $C^{1,\alpha}$ rectifiable. Our goal is to prove a quantitative bound on the H\"older constants in term of the quantity $\sum_i\frac{\alpha_i^2}{2^{-2\alpha i}} < \infty$. For a general one-sided Reifenberg flat set $E$, this means that we can find a parametrization of $E$ via a $C^{1,\alpha}$ map. The proofs of the parametrization results (Sections \ref{prelim} and \ref{improv}) follow the steps of the proof in the paper \cite{davidtoro}.  However detailed knowledge of their paper will not be assumed. Instead specific references will be given for those interested in the proofs of the cited results.

\subsection{Outline of the paper and main results} \label{results}

Throughout the paper, we will prove three different versions of the main theorem on parametrizations. For convenience we will now state only two of them, Theorems \ref{theorema} and \ref{theoremb}. We state the more technical Theorem \ref{epsilonk11} and Theorem \ref{logtheorem} in Section \ref{prelim} after a few more definitions. Then we state Theorems \ref{rectinfty} and \ref{rect1} which are our rectifiability results. 
Let us recall the definition of $\beta$ numbers.
\begin{defin}
Let $E \subseteq \R^n$, $x \in \R^n$, and $r>0$. Let $d$ be a fixed integer, $0 < d < n$. Define
\begin{equation}
\beta_{\infty}^{E,d}(x,r)=\frac{1}{r}\inf_P\left\{\sup_{y \in E \cap B(x,r)} \dist(y,P)\right\}
\end{equation}
if $E \cap B(x,r) \neq \emptyset$, where the infimum is taken over all $d$-planes $P$, and $\beta_{\infty}^E(x,r)=0$ if $E \cap B(x,r) = \emptyset$. 

If $E$ is $\mathcal{H}^d$-measurable and $p \in [1, \infty)$, define
 \begin{equation}
\beta_p^{E,d}(x,r)=\left( \inf_P\left\{\frac{1}{r^d} \int_{y \in E \cap B(x,r)} \left(\frac{\dist(y,P)}{r}\right)^p\,d\mathcal{H}^d(y)\right\} \right)^{\frac{1}{p}},
\end{equation}
for $x \in \R^n$ and $r>0$, where the infimum is taken over all $d$-planes $P$.
\end{defin}

\begin{rk}
In the future we will write $\beta_{\infty}^E(x,r)$ for $\beta_{\infty}^{E,d}(x,r)$, and $\beta_p^E(x,r)$ for $\beta_p^{E,d}(x,r)$, omitting the dependence on $d$, to avoid a too cumbersome notation, as there will not be any chance for confusion.
\end{rk}

Next, we need to define what is meant by one-sided Reifenberg flat.

\begin{defin} \label{reifen}
Let $E \subseteq \R^n$ closed and let $\e >0$. Let $d$ be a fixed integer, $0 < d < n$. Define $E$ to be $(\e,d)$-\emph{Reifenberg flat} if the following condition holds. 

For $x \in E$, $0 < r \leq 10$ there is a $d$-plane $P(x,r)$ passing through $x$ such that 
\begin{align*} \numberthis
\dist(y, P(x,r)) & \leq \e r, \qquad \forall y \in E \cap B(x,r), \\
\dist(y, E) & \leq \e r, \qquad \forall y \in P(x,r) \cap B(x,r).
\end{align*}
\end{defin}

\begin{defin}
Let $x \in \R^n$ and $r>0$. If $E,F \subseteq \R^n$ are non-empty, define \emph{normalized Hausdorff distances} to be the quantities
\begin{equation}
d_{x,r}(E,F) = \frac{1}{r} \max \left\{ \sup_{y \in E \cap B(x,r)} \dist(y,F), \sup_{y \in F \cap B(x,r)} \dist(y,E)\right\}.
\end{equation}
\end{defin}

\begin{defin} \label{reifenhole}
Let $E \subseteq \R^n$
 and let $\e >0$. Define $E$ to be \emph{one-sided $(\e,d)$-Reifenberg flat} if the following conditions (1)-(2) hold. 
\begin{enumerate}
\item[(1)] For $x \in E$, $0 < r \leq 10$ there is a $d$-plane $P(x,r)$ passing through $x$ such that
\[
\dist(y, P(x,r)) \leq \e r, \qquad y \in E \cap B(x,r).
\]
\item[(2)] Moreover we require some compatibility between the $P(x,r)$'s:
\begin{align*} \numberthis
& d_{x,10^{-k}}(P(x,10^{-k}),P(x,10^{-k+1})) \leq \e, \qquad \forall x\in E, \ \forall k \geq 0, \\
&  d_{x,10^{-k+2}}(P(x,10^{-k}),P(y,10^{-k})) \leq \e, \qquad \forall x,y \in E, \ |x-y| \leq 10^{-k+2}, \ \forall k \geq 0.
\end{align*}
\end{enumerate}
\end{defin}
\begin{rk}
We will simply write (one-sided) Reifenberg flat for (one-sided) $(\e,d)$-Reifenberg flat, as $\e$ and $d$ will stay fixed, throughout the paper.
\end{rk}
\begin{rk}
It is important to observe that the sets in Definition \ref{reifen}, for $\e$ sufficiently small, are not allowed to have any holes, meaning that $E$ must be simply connected, while the sets in Definition \ref{reifenhole} are allowed holes of any size. The compatibility conditions is (2) are automatically satisfied by Reifenberg flat sets without holes.
\end{rk}

Before we state our main results, let us recall some theorems of David and Toro \cite{davidtoro}.

\begin{thm} [David-Toro,  Proposition 8.1  \cite{davidtoro}] \label{f}
Let $\e>0$ small enough depending on $n$ and $d$, and let $E \subseteq B(0,1)$, where $B(0,1)$ denotes the unit ball in $\R^n$. Assume $E$ is one-sided Reifenberg flat. Then there exists a map $f \colon \Sigma_0 \to \R^n$, where $\Sigma_0$ is a $d$-plane in $\R^n$, such that $E \subset f(\Sigma_0)$, and $f$ is bi-H\"{o}lder, that is
\begin{equation}
\frac14 |x-y|^{1+C\e} \leq |f(x)-f(y)| \leq 3 |x-y|^{1-C\e}, \quad \text{for all $x,y \in \Sigma_0$},
\end{equation}
where $C$ depends only on $n$ and $d$.

\end{thm}
\begin{rk}
We can get the map $f$ in Theorem \ref{f} to be bi-H\"{o}lder with any exponent strictly smaller than $1$, by choosing $\e$ accordingly small, although sharp exponents are not known.
\end{rk}

Set $r_k=10^{-k}$. 
\begin{thm} [David-Toro, Corollary 12.6 \cite{davidtoro}] \label{dtbinfty}
Let $\e>0$ small enough, and let $E \subseteq B(0,1)$ be a Reifenberg flat set and moreover assume that
\begin{equation} \label{e:dtbinfty}
\sum_{k=0}^{\infty}\beta_{\infty}^E(x,r_k)^2\leq M, \quad \text{ for all $x \in E$}.
\end{equation}
Then $f \colon \Sigma_0 \to \R^n$ is bi-Lipschitz, and $E \subset f(\Sigma_0)$. Moreover the Lipschitz constants depend only on $n$, $d$, and $M$.
\end{thm}
Moreover,
\begin{thm} [David-Toro, Corollary 13.1 \cite{davidtoro}] \label{dtb1}
Let $\e>0$ small enough, and let $E \subseteq B(0,1)$ be a Reifenberg flat $\mathcal{H}^d$-measurable set and moreover assume that
\begin{equation} \label{e:dtb1}
\sum_{k=0}^{\infty}\beta_1^E(x,r_k)^2\leq M, \quad \text{ for all $x \in E$}.
\end{equation}
Then $f \colon \Sigma_0 \to \Sigma$ is bi-Lipschitz. Moreover the Lipschitz constants depend only on $n$, $d$, and $M$. 
\end{thm}

\begin{rk}
In Theorems \ref{dtbinfty} and \ref{dtb1} the set $E$ is require to be Reifenberg flat, while for Theorem \ref{f} $E$ can be merely one-sided flat. This is because we need to use two-sided flatness to control the angles between planes using $\beta$-numbers.

Theorems \ref{theorema} and \ref{theoremb} have more technical counterparts, Theorems \ref{epsilonk11} and \ref{logtheorem}. These results only require one-sided flatness of the set $E$ as they do not utilize its $\beta$-numbers directly, but instead provide a control on the angles which depends on the parametrization itself. Due to the more technical nature of the statements, we will state and prove them in Section \ref{prelim}.
\end{rk}

We are now ready to state our theorems. 

\begin{thmx}\label{theorema}
Let $E \subseteq B(0,1)$ be a Reifenberg flat set and $\alpha \in (0,1)$. Also assume that there exists $M >0$ such that
\begin{equation} \label{e:theorema}
\sum_{k=0}^{\infty}\frac{\beta_{\infty}^E(x,r_k)^2}{r_k^{2\alpha}} \leq M, \quad \text{ for all $x \in E$}.
\end{equation}
Then the map $f \colon \Sigma_0 \to \Sigma$ constructed in Theorem \ref{f} is invertible and differentiable, and both $f$ and its inverse are $C^{1,\alpha}$ maps. In particular, $f$ is continuously differentiable. Moreover the H\"{o}lder constants depend only on $n$, $d$, and $M$.

When $\alpha=1$, if we replace $r_k$ in the left hand side of (\ref{e:theorema}) by $r_k\eta(r_k)$, where $\eta(r_k)^2$ satisfies the Dini condition, then we obtain that $f$ and its inverse are $C^{1,1}$ maps.
\end{thmx}

\begin{rk}
The case $\alpha=0$ was studied in \cite{davidtoro}, see Theorem \ref{dtbinfty}. Notice that they obtain a Lipschitz parametrization, that is $C^{0,1}$ and not a $C^1$ parametrization.

For the case $\alpha=1$ we need a small perturbation of our hypothesis for the proof to extend to this case and obtain a $C^{1,1}$ parametrization. (see Theorem \ref{logtheorem}). 
We say that a function $\omega$ satisfies the Dini condition if $\sum_{k=1}^{\infty} \omega(r_k) < \infty$. A possible choice for $\eta$ in Theorem \ref{theorema} is $\eta(r_k)=\frac{1}{\log(\frac{1}{r_k})^{\gamma}}= \frac{1}{\log(10)^{\gamma}}\frac{1}{k^{\gamma}}$, for $\gamma > \frac12$.
\end{rk}

\begin{rk} \label{e:pointwisebound}
In the case $\alpha=1$ we could also replace (\ref{e:theorema}) and all similar equations below with something of the type
\begin{equation}
\sum_{k=0}^{\infty}\frac{\beta_{\infty}^E(x,r_k)}{r_k} \leq M, \quad \text{ for all $x \in E$},
\end{equation}
as this would only simplify the proof of Theorem \ref{logtheorem}. Notice the lack of the square.

On the other hand, the proof of Theorem \ref{epsilonk11} can also be easily modified to hold if one assumes the possibly simpler condition on $\beta$-numbers:
\begin{equation}
\beta_{\infty}(x,r_k) \leq C r_k^{\alpha} \quad \text{for all $x \in E$}.
\end{equation}
This type of approach has recently been investigated in \cite{delninobinna2019}.

\end{rk}

Even without assuming a higher regularity on our set $E$, such as Ahlfors regularity, we can prove a better sufficient condition involving the possibly smaller $\beta_1$ numbers.

\begin{thmx}\label{theoremb}
Let $E \subseteq B(0,1)$ be a $\mathcal{H}^d$-measurable Reifenberg flat set and $\alpha \in (0,1)$. Also assume that there exists $M > 0$ such that
\begin{equation} \label{e:theoremb}
\sum_{k=0}^{\infty}\frac{\beta_1^E(x,r_k)^2}{r_k^{2\alpha}} \leq M, \quad \text{ for all $x \in E$}.
\end{equation}
Then the map $f \colon \Sigma_0 \to \Sigma$ constructed in Theorem \ref{f} is invertible and differentiable, and both $f$ and its inverse are $C^{1,\alpha}$ maps. In particular, $f$ is continuously differentiable. Moreover the H\"{o}lder constants depend only on $n$, $d$, and $M$.

When $\alpha=1$, if we replace $r_k$ in the left hand side of (\ref{e:theoremb}) by $r_k\eta(r_k)$, where $\eta(r_k)^2$ satisfies the Dini condition, then we obtain that $f$ and its inverse are $C^{1,1}$ maps.
\end{thmx}

Before stating the other results, let us recall the definition of density of a measure. 
\begin{defin}
Let $0 \leq s < \infty$ and let $\mu$ be a Borel measure on $\R^n$. The upper and lower $s$-densities of $\mu$ at $x$ are defined by
\begin{align*} \numberthis
\theta^{*s}(\mu,x)& =\limsup_{r \to 0}\frac{\mu(B(x,r))}{r^s} \\
\theta_{*}^s(\mu,x) &=\liminf_{r \to 0}\frac{\mu(B(x,r))}{r^s}.
\end{align*}
If they agree, their common value is called the $s$-density of $\mu$ at $x$ and denoted by
\begin{equation}
    \theta^s(\mu,x)=\theta^{*s}(\mu,x)=\theta_{*}^s(\mu,x).
\end{equation}
If $E \subseteq \R^n$, we define the upper and lower $s$-densities of $E$ at $x$ as $\theta^{*s}(E,x)=\theta^{*s}(\mathcal{H}^s{\res E},x)$ and  $\theta^{s}_{*}(E,x)=\theta^{s}_{*}(\mathcal{H}^s{\res E},x)$, respectively.
\end{defin}
We are now ready to state the theorems regarding rectifiability. 

\begin{thmz} \label{rectinfty}
Let $E \subseteq \R^n$ be such that $0<\theta^{d*}(E,x)<\infty$, for $\mathcal{H}^d$ a.e. $x \in E$ and let $\alpha\in (0,1)$. Assume that for almost every $x \in E$,
\begin{equation} \label{e:rectinfty}
J^{E}_{\infty,\alpha}(x)=\sum_{k=0}^{\infty} \frac{\beta^E_{\infty}(x,r_k)^2}{r_k^{2\alpha}} < \infty.
\end{equation}
Then $E$ is (countably) $C^{1,\alpha}$ $d$-rectifiable.

When $\alpha=1$, if we replace $r_k$ in the left hand side of (\ref{e:rectinfty}) by $r_k\eta(r_k)$, where $\eta(r_k)^2$ satisfies the Dini condition, then we obtain that $E$ is $C^{2}$ rectifiable.
\end{thmz}

\begin{rk}
For the second part of the statement recall that $C^{1,1}$ rectifiability coincides with $C^2$ rectifiability (see Proposition \ref{c1=c2}).
\end{rk}

\begin{rk}
In Theorem \ref{rectinfty}, we will use the assumptions on the upper density in order to prove that $E$ is rectifiable, using a Theorem of Azzam and Tolsa from \cite{ATII}. We will need rectifiability in order to obtain (local) flatness. Note that, in this case, we cannot weaken the assumptions on the density to be $\theta^{d*}(E,x)>0$ and $\theta^d_*(E,x)<\infty$ to obtain rectifiability, as in \cite{naber}, because we will use that $\theta^{d*}(E,x)<\infty$ to compare $\beta_{\infty}^E$ with $\beta_2^E$ in order to apply the aforementioned theorem of Azzam and Tolsa. See the proof of Theorem \ref{rectinfty} for details.
\end{rk}

We can also state a version of Theorem \ref{rectinfty} for rectifiability of measures. A \emph{Radon measure} $\mu$ on $\R^n$ is a Borel regular outer measure that is finite on compact subsets of $\R^n$.
If $\mu$ is a Radon measure, define
\begin{equation}
\beta_p^{\mu}(x,r)=\inf_P\left\{\frac{1}{r^d} \int_{y \in B(x,r)} \left(\frac{\dist(y,P)}{r}\right)^p\,d\mu(y)\right\}^{\frac{1}{p}},
\end{equation}
for $x \in \R^n$ and $r>0$, where the infimum is taken over all $d$-planes $P$.
Moreover, define
\begin{equation}
J^{\mu}_{p, \alpha}(x)= \sum_{k=0}^{\infty}\frac{\beta_p^{\mu}(x,r_k)^2}{r_k^{2\alpha}}.
\end{equation}

\begin{thmz} \label{rect1}
Let $\mu$ be a Radon measure on $\R^n$ such that $0< \theta^{d *}(\mu,x)$ and $\theta_*^d({\mu},x) <\infty$ for $\mu$-a.e. $x$ and let $\alpha\in (0,1)$. Assume that for $\mu$-a.e. $x \in \R^n$,
\begin{equation} \label{e:rect1}
J^{\mu}_{2,\alpha}(x) < \infty.
\end{equation}
Then $\mu$ is (countably) $C^{1,\alpha}$ $d$-rectifiable. 

When $\alpha=1$, if we replace $r_k$ in the left hand side of (\ref{e:rect1}) by $r_k\eta(r_k)$, where $\eta(r_k)^2$ satisfies the Dini condition, then we obtain that $\mu$ is $C^{2}$ rectifiable.
\end{thmz}

\begin{rk}
The density assumptions in Theorem \ref{rect1} are weaker than the ones in Theorem \ref{rectinfty}, as we will use Theorem \ref{t:ENV} by Edelen, Naber and Valtorta instead of Theorem \ref{azzamtolsa} by Azzam and Tolsa.

Note that the assumption $J^{\mu}_{2,\alpha}(x) < \infty$ implies $J^{\mu}_{1,\alpha}(x)< \infty$ (see Lemma \ref{2implies1}), which is the condition we will need to apply Theorem \ref{theoremb}, and also that $\int_0^1\beta_{\mu,2}(x,r)^2 \frac{dr}{r} < \infty$ which is going to be used to apply a result by Edelen, Naber and Valtorta, \cite{naber} (see Remark \ref{variousbetas} for a more detailed discussion). Also in this case we will use the finiteness of the upper density in Lemma \ref{2implies1}, however, we do not need to assume that as it also follows from Theorem \ref{t:ENV}. 
\end{rk}

\subsection{Plan of the paper}

Because of the technical nature of the proofs of Theorems \ref{theorema} and \ref{theoremb}, in Section \ref{rect} we first prove Theorems \ref{rectinfty} and \ref{rect1} using Theorems \ref{theorema} and \ref{theoremb}. After that, in Section \ref{prelim} we introduce the main tools for the proof and after we state the more technical Theorems \ref{epsilonk11} and \ref{logtheorem}. Then we construct a parametrization for our set $E$ using a so-called coherent collection of balls and planes (CCBP) to then conclude by proving Theorems \ref{epsilonk11} and \ref{logtheorem}. At the end of the Section, we provide proofs of Theorems \ref{theorema} and \ref{theoremb} stated above. 
Finally, in Section \ref{stuff} we include the a few examples, including the one from \cite{anzellottiserapioni}, together with some remarks on the main Theorems.

\subsection{Motivation and related work} \label{motivation}

Peter Jones \cite{atst} proved that, given a collection of points in the plane, we can join them with a curve whose length is proportional to a sum of squares of $\beta$-numbers (plus the diameter). In particular, the length is independent of the number of points. This was the starting point of a series of results seeking to characterize, in a quantitative way, which sets are rectifiable. The motivation came from harmonic analysis, more specifically, the study of singular integral operators. It became clear that the classical notion of rectifiability does not capture quantitative aspects of the operators (such as boundedness) and a quantitative notion of rectifiability was needed. A theory of uniform rectifiability was developed and it turned out that uniformly rectifiable sets are the natural framework for the study of $L^2$ boundedness of singular integral operators with an odd kernel (see   \cite{davidsemmes93, davidsemmes91, tolsabook}).
The theory is developed for sets of any dimension, but a necessary condition for a set to be uniformly rectifiable is that it is $d$-Ahlfors regular, where $d \in \N$. That is, the $d$-dimensional Hausdorff measure of a ball is comparable to its radius to the $d$-th power. 

Jones' Traveling Salesman Theorem works only for $1$-dimensional sets, but does not assume any regularity. Several attempts have been made to prove similar analogues for sets (or measures) of dimension more than $1$, see for instance \cite{pajot1,pajot2}. 
Menger curvature was also introduced to attempt to characterize rectifiability (see, among others, \cite{leger, lerman1, lerman2, polish, blatt, kola2, meurer, othermax, ghinassigoering2019}). Other approaches can be found in   \cite{jessica, delladio, santilli}). 
Azzam and Schul \cite{azzamschul} prove a higher dimensional version of the Traveling Salesman Theorem, that is, they estimate the $d$-dimensional Hausdorff measure of a set using a sum of $\beta$-numbers with no assumptions of Ahlfors regularity. Using this, together with \cite{davidtoro}, Villa \cite{michele} proves a characterization of tangent points of a Jordan curve in term of $\beta$-numbers.

We say that a Radon measure $\mu$ on $\R^n$ is $d$-rectifiable if there exist countably many Lipschitz maps $f_i \colon \R^d \to \R^n$ such that
\[
\mu\left(\R^n \setminus \bigcup_i f_i(\R^d)\right)=0.
\]
Note that a set $E$ is $d$-rectifiable if and only if $\mathcal{H}^d\res_E$ is a $d$-rectifiable measure.

For measures which are absolutely continuous with respect to the Hausdorff measure, the above definition coincides which \emph{Lipschitz graphs} rectifiability. That is, if we require the sets to be almost covered by Lipschitz graphs instead of images, we get an equivalent definition. Garnett, Kilip, and Schul   \cite{GKS} proved that this is not true for general measures, even if we require the doubling condition (that is, the measure of balls is quantitatively comparable if we double the radius). They exhibit a doubling measure supported in $\R^2$, singular with respect to Hausdorff measure, which is \emph{Lipschitz image} rectifiable but is not \emph{Lipschitz graph} rectifiable. 

Preiss, Tolsa, and Toro   \cite{PTT} fully describe the H\"{o}lder regularity of doubling measures in $\R^n$ for measures supported on any (integer) dimension. Badger and Vellis \cite{matt} extended part of the work to lower order rectifiable measures.  They prove that the support of a Radon measure can be parametrized by a $\frac{1}{s}$-H\"{o}lder map, under assumptions on the $s$-dimensional lower density. Badger, Naples and Vellis \cite{mattlisavyron} establish sufficient conditions that ensure a set of points is contained in the image of a $\frac{1}{s}$-H\"older continuous map. Badger and Schul \cite{raananmatt1, raananmatt2} characterize $1$-dimensional (Lipschitz) rectifiable measures in terms of positivity of the lower density and finiteness of a Bishop-Jones type square function. Martikainen and Orponen \cite{orponen} later proved that the density hypothesis above is necessary.

Recently, Edelen, Naber, and Valtorta   \cite{naber} proved that, for an $n$-dimensional Radon measure with positive upper density and finite lower density, finiteness of a Bishop-Jones type function involving $\beta_2$ numbers implies rectifiability. The same authors \cite{naber2} study effective Reifenberg theorems for measures in a Hilbert or Banach space.
Azzam and Tolsa \cite{tolsaI, ATII} characterized rectifiability of $n$-dimensional Radon measures using the same Bishop-Jones type function under the assumption that the upper density is positive and finite. Note that the density condition in \cite{naber} is less restrictive (see \cite{tolsaarxiv}). Tolsa \cite{tolsaarxiv} obtains an alternative proof of the result in \cite{naber} using the techniques from \cite{tolsaI, ATII}. For a survey on generalized rectifiability of measures, including classical results and recent advances, see \cite{survey}.

Kolasi\'{n}ski \cite{kola3} provides a sufficient condition in terms of averaged discrete curvatures, similar to integral Menger curvatures, for a Radon measure with positive lower density and finite upper density to be $C^{1,\alpha}$ rectifiable. Moreover, sharpness of the order of rectifiability of the result is obtained using the aforementioned example from \cite{anzellottiserapioni}. This result is very similar in flavor to the result we prove in this paper. In fact, if the measure is Ahlfors regular, Lerman and Whitehouse \cite{lerman1, lerman2} proved that Menger curvature and a Bishop-Jones type square function involving $L^2$ $\beta$-numbers are comparable on balls. However, for measures which are not Ahlfors regular, the two quantities are not known to be comparable.

Given such distinctions it is natural to investigate different types of rectifiability (e.g., \emph{Lipschitz image} and \emph{Lipschitz graph} rectifiability, $C^2$ and $C^{1,\alpha}$ rectifiability). There has been some progress in this direction concerning rectifiability of sets (by e.g.   \cite{anzellottiserapioni}) but the tools involved rely heavily on the Euclidean structure of $\mathcal{H}^d$ and give qualitative conditions. Dorronsoro \cite{dorronsoro1,dorronsoro2} obtains a characterization for potential spaces and Besov spaces in terms of coefficients which are analogous to higher order versions of Jones's $\beta$-numbers. Several recent works concerning connections between rectifiability and $\beta$-numbers seem to have been inspired by these results.
 There has been a great deal of interest in developing tools which allow further generalizations to rectifiability of measures which provide quantitative results. Using the techniques from   \cite{davidtoro} we develop such tools with the use of $\beta$-numbers and obtain results for $C^{1,\alpha}$ rectifiability. 

Reifenberg \cite{reifen} proves that a ``flat'' set (what is today known as ``Reifenberg flat'' set) can be parametrized by a H\"older map. 
In   \cite{DKT}, David, Kenig, and Toro prove that a $C^{1,\alpha}$ parametrization for Reifenberg flat sets (without holes) with vanishing constants can be achieved under a pointwise condition on the $\beta$'s (their conditions are stronger than our conditions).

Among the results involving Menger curvature, in \cite{polish}, Kolasi\'{n}ski and Szuma\'{n}ska prove that $C^{1,\alpha}$ regularity, with appropriate $\alpha$'s, implies finiteness of functionals closely related to Menger curvature. In \cite{blatt}, Blatt and Kolasi\'{n}ski prove that a compact $C^1$ manifold has finite integral Menger curvatures (a higher dimensional version of Menger curvature) if and only if it can be locally represented by the graph of some Sobolev type map. 

In \cite{kola2}, a bound on Menger curvature together with other regularity assumptions leads to a pointwise bound on $\beta$-numbers: this is the same bound which appears in \cite{DKT}. If in addition the set is \emph{fine}, which among other things implies Reifenberg flatness allowing for small holes (that is, at scale $r$ holes are of the size of $\beta^E_{\infty}(x,r)$), then the same conclusion as in \cite{DKT} holds, that is, the set can be parametrized by a $C^{1,\alpha}$ map.

It is interesting to note that in \cite{DKT} Reifenberg flatness, which does not allow for any holes, is used. On the other hand, in \cite{kola2} they allow small holes, that is, of size  bounded by $\beta$. In contrast, we only require the set to be one-sided Reifenberg flat, which does not impose any restrictions on the size of the holes.

In the last few years, Fefferman, Israel, and Luli   \cite{fefferman} have been investigating Whitney type extension problems for $C^k$ maps, finding conditions to fit smooth functions to data.

\subsection{Further developments}

It is interesting to ask whether there exist analogous necessary conditions for higher order rectifiability. See Section \ref{stuff} for some observations. The author believes similar results for $C^{k,\alpha}$ regularity hold with an appropriate generalization of the Jones $\beta$-numbers and of Reifenberg flatness of higher order. By appropriate generalization we mean to use polynomials instead of $d$-planes to approximate the set. This idea is not new, see for instance \cite{dorronsoro1, dorronsoro2} and, more recently, \cite{marti}, Section 2.2.

\subsection{Acknowledgements}

I am grateful to my advisor Raanan Schul for suggesting the problem and directing me along the way with helpful advice and feedback. I am also thankful for many helpful conversations with Matthew Badger and Mart\'i Prats. Finally, I would like to thank Guy C. David and Jonas Azzam for comments on earlier drafts of this work. 
Some of the work was partially supported by NSF DMS 1361473 and 1763973.

\section{Proof of Theorems \ref{rectinfty} and \ref{rect1} on \texorpdfstring{$C^{1,\alpha}$}{C1alpha} rectifiability} \label{rect}

As mentioned in the introduction, we will start by using Theorems \ref{theorema} and \ref{theoremb} to prove Theorems \ref{rectinfty} and \ref{rect1}. The former will be then proved in the later sections.

 \subsection{A sufficient condition involving \texorpdfstring{$\beta_{\infty}$}{Binfty} numbers}

To prove Theorem \ref{rectinfty} we need to recall a result from \cite{ATII}. 

\begin{thm}[J. Azzam, J. Tolsa, Theorem 1.1,  \cite{ATII}] \label{azzamtolsa}
Let $\mu$ be a finite Borel measure in $\R^n$ such that $0<\theta^{d,*}(\mu,x) < \infty$ for $\mu$-a.e. $x \in \R^n$. If 
\begin{equation}
\int_0^1\beta_{\mu,2}(x,r)^2 \frac{dr}{r} < \infty \quad \text{for $\mu$-a.e. $x \in \R^n$,}
\end{equation}
then $\mu$ is $d$-rectifiable.
\end{thm}

In this section we prove Theorem \ref{rectinfty}.

Before proceeding with the proof we want to note when different Jones' square functions are bounded by each other. 
\begin{rk}
Let us formally state a fact well known by experts in the area and often used in the literature. In the literature, some results prefer using continuous versions of Jones' functions, while others prefer discretized ones. In our statementes we use a discrete version, as in \cite{davidtoro}, but we sometimes relate that to continuous versions, as in \cite{ATII}.

Let $a(r)$ be positive for any $r>0$ and suppose there exist constants $c,C>0$ such that $ca(r_{k+1}) \leq a(r) \leq C a(r_k)$ if $r_{k+1} \leq r \leq r_k$. Then there exists a constant $C_0>0$ such that
\begin{equation}
\frac{1}{C_0}\int_0^1 a(r)\,\frac{dr}{r} \leq \sum_{k=0}^{\infty} a(r_k) \leq C_0  \int_0^1 a(r)\,\frac{dr}{r}.
\end{equation}
\end{rk}

Let us record some of the comparisons between different Jones' functions.

\begin{lem} \label{beta2betainfty}
Let $E \subseteq \R^n$ such that $0<\theta^{d*}(E,x)<\infty$, for a.e. $x \in E$. Set $\mu=\mathcal{H}^d\res E$. If for a.e. $x \in E$
\begin{equation}
J^{E}_{\infty}(x)=\sum_{k=0}^{\infty} \beta^E_{\infty}(x,r_k)^2 < \infty,
\end{equation}
then
\begin{equation}
\int_0^1\beta_{\mu,2}(x,r)^2 \frac{dr}{r} < \infty \quad \text{for $\mu$-a.e. $x \in \R^n$,}
\end{equation}
and hence $E$ is $d$-rectifiable, that is, there exist countably many Lipschitz images $\Gamma_i$ such that $\mathcal{H}^d(E \setminus \bigcup_i \Gamma_i)=0$.
\end{lem}
\begin{proof}
We want to prove that, for a.e. $x \in E$, there exists $r_x>0$ such that if $r<r_x$, then
\begin{equation}
    \beta_{\mu,2}(x,r) \leq C(x) \beta^E_{\infty}(x,r).
\end{equation}
It is enough to prove that, for a.e. $x \in E$, there exists $r_x>0$ such that if $r<r_x$, 
\begin{equation}
    \frac{\mathcal{H}^d(B \cap E)}{r^d} \leq C(x).
\end{equation}
This follows immediately by the assumption $\theta^{d *}(E,x)<\infty$.
The conclusion follows from Theorem \ref{azzamtolsa}.
\end{proof}

\begin{rk}
Note that a set $E$ that satisfies the hypotheses of Theorem \ref{rectinfty} satisfies the hypotheses of Lemma \ref{beta2betainfty}, as  $J^{E}_{\infty}(x) \leq J^{E}_{\infty,\alpha}(x)<\infty$.
\end{rk}
 Let us restate, for convenience of the reader, a Sard-type theorem (Theorem 7.6 in \cite{mattila}).
 \begin{thm} \label{sard}
 If $g \colon \R^d \to \R^n$ is a Lipschitz map, then 
 \begin{equation}
     \mathcal{H}^d(\{g(x) \mid \dim_H (Dg(x)\R^d)< d\})=0.
 \end{equation}
 \end{thm}

\begin{lem} \label{lipccbp}
If $f \colon \R^d \to \R^n$ is a Lipschitz map and $\Gamma=\Im(f)$, then $\Gamma \subseteq \Gamma_b \cup \bigcup_q A_q$, where each $A_q$ is Reifenberg flat and $\mathcal{H}^d(\Gamma_b)=0$.
\end{lem}

\begin{proof}
By Theorem 3.2.29 in \cite{federer} (Lipschitz and $C^1$ rectifiability are equivalent notions for measures absolutely continuous to Hausdorff measure), we know that there exists countably many $C^1$ maps $g_i \colon \R^d \to \R^n$ such that $\Gamma \subseteq \bigcup_i g_i(\R^d)$. 
To simplify notations, let $g=g_i$, for some $i$, for the time being. For $\mathcal{H}^d$-almost every $z \in \Im(g)$, we know by Theorem \ref{sard} that $\operatorname{rank}(Dg(x))=d$ where $x$ is such that $g(x)=z$. 

Because $g$ is a continuously differentiable map, for any $\e'>0$, we know that there exists a small enough neighborhood $U_z \ni x$ such that $\operatorname{rank}(Dg(y))=d$ and, by the Implicit Function Theorem, $g(U_z)$ is a $C^1$ graph over $U_z$.
\begin{equation}
    |Dg(x)-Dg(y)| < \e'
\end{equation}
for every $y \in U_z$. We want to prove that $g(U_z)$ is Reifenberg flat. 
For any $x \in g(U_z)$ and $r>0$ let $P_{x,r}$ be the unique tangent $d$-plane to $g(U_z)$ at $x$. 
We need to check that  
\begin{align*} \numberthis
& \dist(y, P_{x,r}) \leq \e r, \qquad y \in g(U_z) \cap B(x,r), \\
&  \dist(y, g(U_z)) \leq \e r, \qquad y \in P_{x,r} \cap B(x,r).
\end{align*}

By choosing $\e'>0$ above small enough with respect to $\e$, all conditions are satisfied, as the derivative varies smoothly and so do the planes $P_{x,r}$'s, and we recover a two-sided control as locally, we have a $C^1$ graph.

Because the choices of $g_i$ and $z$ are arbitrary we can repeat the same procedure for all the maps. Note we can choose countably many $z_l$ and still obtain a cover for $g_i(\R^d)$. We then have a collection of neighborhoods $U_{z_l}^i$ so that each $g_i(U_{z_l}^i)$ is Reifenberg flat and $\Gamma \subseteq \bigcup_{i,x} g_i(U_{z_l}^i)$ up to $\mathcal{H}^d$ measure zero $\Gamma_b$. Re-indexing the collection by $A_q$, we obtain the desired result.
\end{proof}

We are now ready to prove Theorem \ref{rectinfty}.

\begin{proof}[Proof of Theorem \ref{rectinfty}]
By Lemma \ref{beta2betainfty} there exists countably many Lipschitz images $\Gamma_i$ such that $\mathcal{H}^d(E \setminus \bigcup_i \Gamma_i)=0$. 
Let $E_{i,q}=E\cap (\Gamma_i)_q$, where we applied Lemma \ref{lipccbp} to each $\Gamma_i$ and obtained $A_q=(\Gamma_i)_q$. Now, define 
\begin{equation}
    E_{i,q,p}=\left\{ x \in E_{i,q} \mid J_{\infty,\alpha}(x) \leq p \right\}.
\end{equation}
By Lemma \ref{lipccbp} each of the $E_{i,q,p}$ satisfies the hypotheses of Theorem \ref{theorema} and hence it can be parametrized by a $C^{1,\alpha}$ surface.
Because $E=E_b \cup \bigcup_{i,q,p} E_{i,q,p}$, where $E_b$ has $\mathcal{H}^d$-measure zero, Theorem \ref{rectinfty} follows.
\end{proof}

\subsection{A sufficient condition involving \texorpdfstring{$\beta_{2}$}{B1} numbers}

We can also state a version of Theorem \ref{rectinfty} for rectifiability of measures, that is Theorem \ref{rect1}. 
If $\mu$ is a Radon measure, and $1\leq p < \infty$, define
\begin{equation}
\beta_p^{\mu}(x,r)=\inf_P\left\{\frac{1}{r^d} \int_{y \in B(x,r)} \left(\frac{\dist(y,P)}{r}\right)^p\,d\mu(y)\right\}^{\frac{1}{p}},
\end{equation}
for $x \in \R^n$ and $r>0$, where the infimum is taken over all $d$-planes $P$.
Moreover, define,
\begin{equation}
J^{\mu}_{p, \alpha}(x)= \sum_{k=0}^{\infty}\frac{\beta^{\mu}_p(x,r_k)^2}{r_k^{2\alpha}}.
\end{equation}

To prove the theorem we will use the following theorem by N. Edelen, A. Naber and D. Valtorta.

\begin{thm}[N. Edelen, A. Naber and D. Valtorta, \cite{naber}] \label{t:ENV}
Let $\mu$ be a finite Borel measure in $\R^n$ such that $0< \theta^{d *}(\mu,x)$ and $\theta_*^d({\mu},x) <\infty$ for $\mu$-a.e. $x \in \R^n$. If 
\begin{equation}
\int_0^1\beta_{\mu,2}(x,r)^2 \frac{dr}{r} < \infty \quad \text{for $\mu$-a.e. $x \in \R^n$,}
\end{equation}
then $\mu$ is $d$-rectifiable and $\theta^{d *}(\mu,x)< \infty$, for $\mu$-a.e. $x \in \R^n$.
\end{thm}

\begin{rk} \label{variousbetas}
Condition \ref{e:rect1} in Theorem \ref{rect1} is slightly stronger than what we actually need. In fact, it implies that $J^{\mu}_{1,\alpha}(x)< \infty$ (see Lemma \ref{2implies1} below). We use the latter condition to apply Theorem \ref{theoremb}. It also implies that $\int_0^1\beta_{\mu,2}(x,r)^2 \frac{dr}{r} < \infty$, which is a necessary hypothesis for applying Theorem \ref{t:ENV}. Notice that assuming only boundedness of the $L^1$ Bishop-Jones square function would not guarantee the set to be rectifiable (see \cite{tolsaarxiv}).

As observed in the introduction, the density assumptions of Theorem \ref{t:ENV} are weaker than the ones in Theorem \ref{azzamtolsa}. Note again that, if $\mu$ is rectifiable then it has $0< \theta_*^d({\mu},x)$ $\mu$-almost everywhere (for a proof, see \cite{raananmatt1}), so the following lemmas apply to $\mu$ in Theorem \ref{rect1}.
We will use the fact that $0< \theta^d_{*}(\mu,x)$ in order to be able to compare $\beta$-numbers computed with respect to $\mu$ and those computed using $\mathcal{H}^d$ and the fact that $\theta^{d *}(\mu,x) < \infty$ to compare $L^1$ and $L^2$ Jones functions.
\end{rk}

\begin{lem} \label{2implies1}
Let $\mu$ be a Radon measure on $\R^n$ and let $x$ such that $\theta^{d *}(\mu,x)<\infty$ and $J^{\mu}_{2,\alpha}(x) < \infty$. Then, $J^{\mu}_{1,\alpha}(x) < \infty$.
\end{lem}
\begin{proof}
It is enough to prove there exists $r_x >0$ such that if $r<r_x$,
\begin{equation}
    \beta_{\mu,1}(x,r) \leq C(x) \beta_{\mu,2}(x,r).
\end{equation}
By H\"{o}lder's inequality we get
\begin{equation}
    \frac{1}{r^d}\int_{B(x,r)}\frac{d(y,P)}{r}\,d\mu(y) \leq \left(\frac{\mu(B(x,r))}{r^d}\right)^{\frac12}\left(\frac{1}{r^d}\int_{B(x,r)}\left(\frac{d(y,P)}{r}\right)^2\,d\mu(y)\right)^{\frac12}.
\end{equation}
Because $\theta^{d *}(\mu,x)<\infty$, we get $\frac{\mu(B(x,r))}{r^d}\leq C(x)$ and we are done.
\end{proof}

We would like to proceed as in the proof of Theorem \ref{rectinfty}. Because of our assumptions (see Remark \ref{variousbetas}), it follows from Theorem \ref{t:ENV} that $\mu$ is $d$-rectifiable, that is, there exist countably many Lipschitz images $\Gamma_i$ such that $\mu\left(E \setminus \bigcup_i \Gamma_i\right)=0$. 

Let $E=\supp\mu \cap \{x \in \R^n \mid J^{\mu}_{2,\alpha}(x) < \infty\}$, where $\supp\mu=\overline{\{x \in \R^n \mid \mu(B(x,r)>0 \text { for all $r>0$})\}}$ denotes the support of the measure $\mu$. From Lemma \ref{lipccbp} we get that each $E_{i,q}=E \cap (\Gamma_i)_q$ is Reifenberg flat. To apply Theorem \ref{theoremb} we need to ensure that the ``Euclidean'' $\beta_1$ numbers (i.e. the $\beta_1$ numbers computed with respect to the $d$-dimensional Hausdorff measure) satisfy the hypothesis of Theorem \ref{theoremb}.

\begin{lem}
Let $E_{i,q}$ be as above. There exists a countable collections of subsets $E_{i, q, N, m}$ such that for every $x \in E_{i, q, N, m}$ there exist numbers $C_x>0$ and $r_x>0$ such that for every $r_k<r_x$ we have
\begin{equation}
\sum_{\stackrel{k}{r_k<r_x}}\frac{  \beta_1^{E_{i, q, N, m}}(x,r_k)^2}{r_k^{2\alpha}} \leq C_x.
\end{equation}
\end{lem}

\begin{proof}
By our assumptions on $\mu$ we know that for every $x \in E_{i, q}$ there exist numbers $C_x>0$ and $r_x>0$ such that for every $r_k<r_x$ we have
\begin{equation}
\sum_{\stackrel{k}{r_k<r_x}} \frac{  \beta_1^{\mu\res E_{i, q}}(x,r)^2}{r_k^{2\alpha}} \leq C_x.
\end{equation}

Define $E_{i, q, N, m}$ by
\begin{equation}
E_{i, q, N, m}=\left\{x \in E_{i,q} \mid  \frac{1}{N} \leq \frac{\mu(B(x,r) \cap E_{i,q})}{r^d}\leq N \ \text{for $r< 2^{-m}$}\right\}.
\end{equation}

In order to prove the statement it is enough to prove that each $\beta_1^{E_{i, q, N, m}}(x,r)$ is bounded above by a constant multiple of $\beta_1^{\mu\res E_{i, q}}(x,r)$. To obtain this, it is enough to prove that, for some constant $C$, we have 
\begin{equation}
    \mathcal{H}^d(E_{i, q,N,m}\cap B)\leq CN \mu(E_{i,q}\cap B).
\end{equation}
This follows from Theorem 6.9(2) in \cite{mattila}. 
\end{proof}

Finally, define
\begin{equation}
   E_{i,q,N,m,p}=\left\{x \in E_{i,q,N,m} \mid J_{1,\alpha}(x) \leq p\right\}.
\end{equation}
From the results above the following lemma follows immediately.
\begin{lem}
Each $E_{i,q,N,m,p}$ satisfies the hypotheses of Theorem \ref{theoremb} and hence it can be parametrized by a $C^{1,\alpha}$ surface.
\end{lem}

Now, we have that $E=E_b \cup \bigcup_{i,q,N,m,p} E_{i,q,N,m,p}$, where $E_b$ has $\mathcal{H}^d$-measure zero, by Lemma \ref{lipccbp}, the definition of rectifiability, and continuity from below. The lemma below proves that $E_b$ has also $\mu$ measure zero, so Theorem \ref{rect1} follows.

\begin{lem}
Let $A \subset \R^n$ and $\nu$ a Radon measure such that $\theta^{d *}(\nu,x)<\infty$ for $\nu$-a.e. $x$. If $\mathcal{H}^d(A)=0$, then $\nu(A)=0$.
\end{lem}
The lemma follows immediately from Theorem 6.9(1) in \cite{mattila}. 

\section{The more technical result on parametrization} \label{prelim}

We now proceed to introduce the main tools for the proofs of Theorems \ref{theorema} and \ref{theoremb}. In this section, we will construct the map $f$ and obtain distortion estimates for it. We then relate those estimates to our conditions on $\beta$-numbers, and prove the main theorems, Theorems \ref{theorema} and \ref{theoremb}.

\subsection{More definitions and statement of the more technical result}

Given a one-sided Reifenberg flat set, we now want to construct a so-called \emph{coherent collection of balls and planes} (CCBP) for $E$ (for more details see the discussion after Theorem 12.1 in \cite{davidtoro}).

Let $E$ be as above and set $r_k=10^{-k}$. Choose a maximal collection of points $\{x_{j,k}\} \subset E$, $j \in J_k$ such that $|x_{i,k}-x_{j,k}| \geq r_k$, for $i,j \in J_k$, $i \neq j$. Let $B_{j,k}$ be the ball centered at $x_{j,k}$ with radius $r_k$. For $\lambda > 1$, set 
\begin{equation}
V_k^{\lambda}=\bigcup_{j \in J_k} \lambda B_{j,k}.
\end{equation}

Because of our assumptions on the set $E$ we can assume that the initial points $\{x_{j,0}\}$ are close to a $d$-plane $\Sigma_0$, that is $\dist(x_{j,0}, \Sigma_0) \leq \e$, for $j \in J_0$.
Moreover, for each $k \geq 0$ and $j \in J_k$ we assume that there exists a $d$ plane $P_{j,k}$ through $x_{j,k}$ such that
\begin{align} 
 d_{x_{j,k}, 100r_k}(P_{i,k}, P_{j,k}) & \leq \e \text{ for $k \geq 0$ and $i,j \in J_k$ such that $|x_{i,k}-x_{j,k}| \leq 100r_k$,}  \label{ccbp1} \\
 d_{x_{i,0},100}(P_{i,0},\Sigma_0) & \leq \e \text{ for $i \in J_0$,} \\
 d_{x_{i,k},20r_k}(P_{i,k},P_{j,k+1})&  \leq \e \text{ for $k \geq 0$, $i \in J_k$ and $j \in J_{k+1}$ s.t. $|x_{i,k}-x_{j,k+1}| \leq 2r_k$.}
\end{align}

\begin{defin} \label{defccbp}
A \emph{coherent collection of balls and planes} for $E$ is a pair $(B_{j,k},P_{j,k})$ with the properties above. We assume that $\e>0$ is small enough, depending on $d$ and $n$.
\end{defin}

We will use this collection to construct the parametrization, as explained in the following section. 
We now define the coefficients $\e_k$ which differ from classic $\beta$-numbers in that they take into account neighbouring points at nearby scales. In section \ref{improv} the relationship between the two will be made explicit.
\begin{defin} \label{e_k}
For $k \geq 1$ and $y \in V^{10}_k$ define
\begin{equation}
 \e_k(y)=\sup\{ d_{x_{i,l},100r_l}(P_{j,k},P_{i,l}) \mid j \in J_k, \ l \in \{k-1,k\}, i \in J_l \  \ y \in 10B_{j,k} \cap11B_{i,k}\}
\end{equation}
and $ \e_k(y)=0$, for $y \in \R^n \setminus V_k^{10}$.
\end{defin}

As in \cite{davidtoro} $f$ will be constructed as a limit. To construct the sequence we need a smooth partition of unity subordinate to $\{B_{j,k}\}$.
Following the construction in Chapter 3 of \cite{davidtoro}, we can obtain functions $\theta_{j,k}(y)$ and $\psi_k(y)$ such that each $\theta_{j,k}$ is nonnegative and compactly supported in $10B_{j,k}$, and $\psi_k(y)=0$ on $V_k^8$. Moreover we have, for every $y \in \R^n$,
\begin{equation} 
\psi_k(y) + \sum_{j \in J_k} \theta_{j,k}(y)\equiv 1.
\end{equation}
Note that, because $\psi_k(y)=0$ on $V_k^8$, this means that 
\begin{equation} \label{partition}
\sum_{j \in J_k} \theta_{j,k}(y)\equiv 1, \quad \text{for every $y \in V_k^8$}.
\end{equation}
Finally we have that
\begin{equation}
|\nabla^m\theta_{j,k}(y)| \leq C_m\frac{1}{r_k^m},  \qquad |\nabla^m\psi_k(y)| \leq C_m\frac{1}{r_k^m}.
\end{equation}
 
Following \cite{davidtoro}, our plan is to define a map $f$ on a $d$-plane $\Sigma_0$. We define $f \colon \R^n \to \R^n$ and later on we will only care about its values on $\Sigma_0$. With a slight abuse of notation we will still denote the restricted map to $\Sigma_0$ as $f$. We define the sequence $\{f_k \colon \R^n \to \R^n\}$ inductively by
\begin{equation}
f_0(y)=y \qquad \text{and} \qquad f_{k+1}=\sigma_k \circ f_k,
\end{equation}
where
\begin{equation} \label{sigma}
\sigma_k(y)= \psi_k(y)y + \sum_{j \in J_k} \theta_{j,k}(y)\pi_{j,k}(y).
\end{equation}
where $\pi_{j,k}$ denotes the orthogonal projection from $\R^n$ to $P_{j,k}$. In the future we denote by $\pi^{\perp}_{j,k}$ the projection onto the $(n-d)$-plane perpendicular to $P_{j,k}$ (passing through the origin). 
Next, we observe that the $f_k$'s converge to a continuous map $f$. We include below the proof of this fact from \cite{davidtoro}.
Note that 
\begin{equation}
|\sigma_k(y)-y| \leq 10r_k \quad \text{for $y \in \R^n$}
\end{equation}
because $\sum_{j \in J_k}\theta_{j,k}(y) \leq 1$ and $|\pi_{j,k}(y)-y| \leq 10r_k$ when $\theta_{j,k}(y)\neq 0$ ($\theta_{j,k}$ is compactly supported in $10B_{j,k}$, so that means $y \in 10B_{j,k}$. This implies that
\begin{equation}
    \|f_{k+1}-f_k\|_{\infty} \leq 10r_k
\end{equation}
so that the maps $f_k$'s converge uniformly on $\R^n$ to a continuous map $f$.

\begin{thm} [G. David, T. Toro,  Proposition 8.3 \cite{davidtoro}] \label{ftdt}
Let $\e >0$ and $E$ as above. If we also assume that there exists $M >0$ such that
\begin{equation}
\sum_{k=0}^{\infty}  \e_k(f_k(z))^2 \leq M, \qquad \text{ for all $z \in \Sigma_0$}.
\end{equation}
then the map $f \colon \Sigma_0 \to \Sigma$ constructed in Theorem \ref{f} is bi-Lipschitz. Moreover the Lipschitz constants depend only on $n$, $d$, and $M$.
\end{thm}

As mentioned before, we are interested in finding a condition on the $ \e_k$'s to improve the results on the map $f$. The theorems we want to prove are the following.

\begin{thm} \label{epsilonk11}
Let $E \subseteq B(0,1)$ as above, with $\e>0$ small enough, and $\alpha \in (0,1)$. Also assume that there exists $M >0$ such that
\begin{equation} \label{beta2r_alpha}
\sum_{k=0}^{\infty} \frac{ \e_k(f_k(z))^2}{r_k^{2\alpha}} \leq M, \qquad \text{ for all $z \in \Sigma_0$}.
\end{equation}
Then the map $f \colon \Sigma_0 \to \Sigma$ constructed in Theorem \ref{f} is invertible and differentiable, and both $f$ and its inverse have $\alpha$-H\"older directional derivatives. In particular, $f$ is continuously differentiable. Moreover the H\"older constants depend only on $n$, $d$, and $M$.
\end{thm}

\begin{rk}
We will define $f \colon \R^n \to \R^n$ but we are only interested in its values on $\Sigma_0$ and $\Sigma=f(\Sigma_0)$. The directional derivatives for the inverse are derivatives along directions tangent to $\Sigma$.
\end{rk}

\begin{thm} \label{logtheorem}
Let $E \subseteq B(0,1)$ as above, with $\e>0$ small enough and let $\eta(r_k)^2$ satisfy the Dini condition. Also assume that there exists $M >0$ such that
\begin{equation} \label{beta2r}
\sum_{k=0}^{\infty} \left( \frac{\e_k(f_k(z))}{r_k\eta(r_k)}\right)^2 \leq M, \qquad \text{ for all $z \in \Sigma_0$}.
\end{equation}
Then the map $f \colon \Sigma_0 \to \Sigma$ constructed in Theorem \ref{f} is invertible and differentiable, and both $f$ and its inverse have Lipschitz directional derivatives. In particular, $f$ is continuously differentiable. Moreover the Lipschitz constants depend only on $n$, $d$, and $M$.
\end{thm}

\subsection{Estimates on the parametrization}

We now want to collect estimates on the derivatives of the $\sigma_k$'s. Recall, by (\ref{sigma}), we defined  $\sigma_k(y)= \psi_k(y)y + \sum_{j \in J_k} \theta_{j,k}(y)\pi_{j,k}(y)$.

\begin{rk} \label{notation}
We set up some notation for the derivatives. Below $D$ and $D^2$ will denote slightly different things depending on the map they are applied to. 
\begin{itemize}
\item For the partition of unity $\theta_{j,k}, \psi_k \colon \R^n \to \R$, $D\theta_{j,k}$ and $D\psi_k$ denote the usual gradient, that is an $n$-vector, that is, a $n \times 1$ matrix. $D^2 \theta_{j,k}$ and $D^2 \psi_k$ denote the Hessian, which is a $n \times n$ matrix.
\item For vector valued maps $g \colon \R^n \to \R^n$, such as $f, f_k, \sigma_k, \pi_{j,k}, \pi_{j,k}^{\perp}$, write $g=(g^1, \dots, g^n)$, where the $g^i$ are the coordinate functions. Then $Dg=(Dg^1, \dots, Dg^n)$ which can be looked at as an $n \times n$ matrix. Similarly, $D^2g=(D^2g^1, \dots, D^2g^n)$ is a $3$-tensor, that is a bilinear form $\R^n \times \R^n \to \R^n$ that acts on vector $u,v \in \R^n$ via $D^2g\cdot u \cdot v =(D^2g^1\cdot u \cdot v, \dots, D^2g^n\cdot u \cdot v)$.
\end{itemize}
 In what follows $|\cdot|$ denotes the standard Euclidean norm on $\R^N$, for the appropriate $N$ (where we have identified $M_{n\times n}$ with $\R^{n^2}$).
\end{rk}

\begin{rk} \label{d2=0}
Note that while $\pi_{j,k}$ is an affine map, $\pi^{\perp}_{j,k}$ is a linear map. Also note that $D\pi_{j,k}(y)$, the Jacobian of $\pi_{j,k}$  at $y \in \R^n$, is the orthogonal projection onto the $d$-plane parallel to $P_{j,k}$ passing through the origin. Note that the Hessian $D^2\pi_{j,k}(y)=0$, for all $y \in \R^n$.  
\end{rk}

By differentiating (\ref{sigma}), we get that for $y \in V_k^{10}$, we have
\begin{equation} \label{dsigma}
D\sigma_k(y)= \psi_k(y)I + \sum_{j \in J_k} \theta_{j,k}(y)D\pi_{j,k} + yD\psi_k(y) + \sum_{j \in J_k} \pi_{j,k}(y) D\theta_{j,k}(y).
\end{equation}

Note that if $y \notin V_k^{10}$, then $\sigma_k(y)=y$ and also $D\sigma_k(y)=I$. Then we also have $D^2\sigma_k(y)=0$.

\begin{lem} \label{idk}
Let $y \in V_k^{10}$. We have
\begin{equation}  \label{ddsigma}
D^2\sigma_k(y) = 2 D\psi_k(y)I  + 2\sum_{j \in J_k} D\theta_{j,k}(y)D\pi_{j,k} + y D^2\psi_k(y) + \sum_{j \in J_k} \pi_{j,k}(y)D^2\theta_{j,k}(y).
\end{equation}

Choose $i=i(y) \in J_k$ such that $y \in 10B_{i,k}$ and set
\begin{equation}
g(y)= 2 D\psi_k(y)D\pi^{\perp}_{i,k}+ (y-\pi_{i,k}(y))D^2\psi_k(y). 
\end{equation}
Then
\begin{equation} \label{d2g}
\left|D^2\sigma_k(y) - g(y)\right|  \leq C \frac{\e}{r_k},
\end{equation}
where $C>0$ is a constant.
\end{lem}

\begin{proof}
We obtain (\ref{ddsigma}) by differentiating (\ref{dsigma}). For the last statement, recalling (\ref{partition}), we have
\begin{align*} \numberthis
g(y) & =  2 D\psi_k(y)D\pi^{\perp}_{i,k}+ (y-\pi_{i,k}(y))D^2\psi_k(y) = \\
& = 2 D\psi_k(y)[I-D\pi_{i,k}]+ yD^2\psi_k(y)-\pi_{i,k}(y)D^2\psi_k(y) = \\
& = 2 D\psi_k(y)I + 2 \sum_{j\in J_k} D\theta_{j,k}(y)D\pi_{i,k} + yD^2\psi_k(y) + \sum_{j\in J_k} \pi_{i,k}(y)D^2\theta_{j,k}(y).
\end{align*}
Now, note that $\left| D^2\theta_{j,k}(y)\right| \leq C\frac{1}{r_k^2}$. Moreover by (\ref{ccbp1}), for all nonzero terms, we have $\left| D\pi_{j,k}-D\pi_{i,k}\right| \leq C \e$, because $\theta_{j,k}=0$ outside of $10B_{j,k}$, so that $y \in 10B_{j,k}$ and hence $|x_{i,k}-x_{j,k}|<100r_k$ for our choice of $(i,k)$.
Hence, we get
\begin{align*} \numberthis
\left|D^2\sigma_k(y) - g(y)\right| & \leq 2\sum_{j \in J_k} \left|D\theta_{j,k}(y)\right| \left| D\pi_{j,k}-D\pi_{i,k}\right|  + \sum_{j \in J_k} \left| D^2\theta_{j,k}(y) \right| \left|\pi_{j,k}(y)-\pi_{i,k}(y)\right| \leq \\
& \leq C\frac{1}{r_k} \cdot C \e + C\frac{1}{r_k^2} \cdot C \e r_k = \\
& = C\frac{\e}{r_k},
\end{align*}
where we used the fact that $\left|\pi_{j,k}(y)-\pi_{i,k}(y)\right| \leq C \e r_k$ by (\ref{ccbp1}).
\end{proof}

%
%

We now want to collect some more estimates. Let $\Sigma_k$ be the image of $\Sigma_0$ under $f_k$, i.e. $\Sigma_k=f_k(\Sigma_0)=\sigma_{k-1} \circ \cdots \circ \sigma_0 (\Sigma_0)$. First, we need to recall some results from   \cite{davidtoro}. 
The main result is a local Lipschitz description of the $\Sigma_k$'s. For convenience we introduce the following notation for boxes.
\begin{defin} [Chapter 5, \cite{davidtoro}]
If $x \in \R^n$, $P$ is a $d$-plane through $x$ and $R>0$, we define the box $D(x,P,R)$ by
\begin{equation}
D(x,P,R)= \left\{ z + w \mid z \in P \cap B(x,R) \text{ and } w \in P^{\perp} \cap B(0,R)\right\}.
\end{equation}
\end{defin}

Recall that for a Lipschitz map $A \colon P \to P^{\perp}$ the graph of $A$ over $P$ is $\Gamma_A=\{z + A(z) \mid z \in P\}$.

\begin{prop} [Proposition 5.1 \cite{davidtoro}] \label{chap5dt}
For all $k \geq 0$ and $j \in J_k$, there is a Lipschitz function $A_{j,k} \colon P_{j,k} \cap 49B_{j,k} \to P_{j,k}^{\perp}$ of class $C^2$, $|A_{j,k}(x_{j,k})|\leq C \e r_k$, with 
\begin{equation}
\left| DA_{j,k}(z)\right| \leq C \e, \qquad z \in P_{j,k} \cap 49B_{j,k},
\end{equation}
such that around $x_{i,j}$ $\Sigma_k$ coincides with the graph of $A_{j,k}$, that is
\begin{equation} \label{graph}
\Sigma_k \cap D(x_{j,k},P_{j,k},49r_k)=\Gamma_{A_{j,k}}\cap D(x_{j,k},P_{j,k},49r_k).
\end{equation}
Moreover, we have that 
\begin{equation}
|\sigma_k(y)-y| \leq C \e r_k \text{ for $y \in \Sigma_k$}
\end{equation}
and
\begin{equation}
\left| D\sigma_k(y)-D\pi_{j,k}-\psi_k(y)D\pi^{\perp}_{j,k}\right| \leq C\e \text{ for $y \in \Sigma_k \cap 45B_{j,k}$}.
\end{equation}
\end{prop}
Proposition \ref{chap5dt} provides a small Lipschitz graph (that is, a Lipschitz graph with a small constant) description for the $\Sigma_k$ around $x_{j,k}$. Note that, away from $x_{j,k}$, $\sigma_k=\operatorname{id}$, so that $\Sigma_k$ stays the same so that it is not hard to get control there too. 
The proof of Proposition \ref{chap5dt} is quite long and involved, and proceeds by induction. For $k=0$, $\Sigma_0$ is a plane, and because $P_{j,k}$ and $P_{i,k+1}$ make small angles with each other, once we have a Lipschitz description of $\Sigma_k$ we can obtain one with a comparable constant for $\Sigma_{k+1}$.
Using Proposition \ref{chap5dt} we can get estimates on the second derivatives of the $\sigma_k$'s.

\begin{prop} \label{chap5}
For all $k \geq 0$, $j \in J_k$, $y \in \Sigma_k \cap 45B_{j,k}$,  we have
\begin{equation}
\left|D^2\sigma_k(y) - 2 D\psi_k(y)D\pi^{\perp}_{j,k}\right| \leq C\frac{\e}{r_k}.
\end{equation}
\end{prop}

\begin{proof}
Let $j \in J_k$ and $y \in \Sigma_k \cap 45B_{j,k}$ be given. If $y \notin V_k^{10}$, then $\psi_k(y)=1$ and $D^2\sigma_k(y)=0$, so there is nothing to prove. So we may assume that $y \in V_k^{10}$ and choose $i \in J_k$ such that $|y-x_{i,k}| \leq 10r_k$. Recall that, by (\ref{d2g}) in Lemma \ref{idk},
\begin{equation}
\left|D^2\sigma_k(y) - g(y)\right|  \leq C\frac{\e}{r_k}.
\end{equation}

We want to control
\begin{align*} \numberthis
B & =g(y)- 2 D\psi_k(y)D\pi^{\perp}_{j,k} = \\
&= 2 D\psi_k(y)[D\pi^{\perp}_{i,k}-D\pi^{\perp}_{j,k}] + [y-\pi_{i,k}(y)]D^2\psi_k(y)
\end{align*}
In the construction of the coherent families of balls and planes, since $y \in 45B_{j,k} \cap 10B_{i,k}$, (\ref{ccbp1}) says that
\begin{equation}
d_{x_{j,k}, 100r_k}(P_{i,k}, P_{j,k})  \leq \e
\end{equation}
and so, 
\begin{equation}
\left|D\pi_{i,k}-D\pi_{j,k} \right| + \left| D\pi^{\perp}_{i,k}-D\pi^{\perp}_{j,k}\right| \leq C \e.
\end{equation}
Recalling also that $\left| D\psi_k(y)\right| \leq C\frac{1}{r_k}$, we can bound the first two terms of $B$ by  $C\frac{\e}{r_k}$. Next
\begin{align*} \numberthis
[y-\pi_{i,k}(y)]D^2\psi_k(y) & 
\leq C r_k^{-2} |y-\pi_{i,k}(y)| = \\
& = C r_k^{-2} \dist(y,P_{i,k}) \leq \\
& \leq C r_k^{-2} \dist(y,P_{j,k}) + C\frac{\e}{r_k}.
\end{align*}
By the results in Proposition \ref{chap5dt}, we also have
\begin{equation}
\dist(y,P_{j,k}) \leq |A_{j,k}(x_{j,k})| + C \e r_k \leq C\e r_k.
\end{equation}
Then, finally, 
\begin{equation}
\left|D^2\sigma_k(y) - 2 D\psi_k(y)D\pi^{\perp}_{j,k}\right| \leq \left|D^2\sigma_k(y) - g(y)\right| + |B| \leq C \frac{\e}{r_k}. \qedhere
\end{equation}
\end{proof}

In the next lemmas from \cite{davidtoro} we want to check how much the mappings $f_k$ distort lengths and distances. We are only concerned with directions parallel to the tangent planes to $\Sigma_k$. Lemma \ref{tgderivatives} below is enough to obtain the original H\"{o}lder estimates in Theorem \ref{f}, but we need more precise estimates to obtain more quantitative results.

\begin{lem} [Lemma 7.1 \cite{davidtoro}] \label{tgderivatives}
Let $k \geq 0$, $\sigma_k \colon \Sigma_k \to \Sigma_{k+1}$ is a $C^2$ diffeomorphism, and for $y \in \Sigma_k$
\begin{equation}
D\sigma_k(y) \colon T \Sigma_k(y) \to T\Sigma_{k+1}(\sigma_k(y)) \text{ is a $(1+C\e)$-bi-Lipschitz map}.
\end{equation}
Moreover, for $v \in T\Sigma_k(y)$
\begin{equation} \label{inverse2}
\left| D\sigma_k(y)\cdot v - v \right| \leq  C \e|v|.
\end{equation}
\end{lem}

Recall Definition \ref{e_k}, 
\begin{equation}
 \e_k(y)=\sup\{ d_{x_{i,l},100r_l}(P_{j,k},P_{i,l}) \mid j \in J_k, \ l \in \{k-1,k\}, i \in J_l \  \ y \in 10B_{j,k} \cap11B_{i,k}\}
\end{equation}
and $ \e_k(y)=0$, for $y \in \R^n \setminus V_k^{10}$.
The numbers $\e_k$ measure the angles between the planes $P_{j,k}$ and $P_{i,l}$ and, while we know that $\e_k(y)\leq \e)$ by definition of CCBP we want to keep track of the places where they are much smaller and improve the estimates obtained before. 

The next lemma provides improved distortion estimates for the tangential derivatives of $\sigma_k$, which will be useful when estimating $|f(x)-f(y)|$.
For two $d$-planes $P_1$ and $P_2$, with $0 \in P_1 \cap P_2$, we define the angle between them to be 
\begin{equation}
\operatorname{Angle}(P_1,P_2)=\sup_{y \in P_1, |y|=1}\dist(y,P_2)=\sup_{x \in \R^n, |x|=1}|\pi_1^{\perp}\circ \pi_2(x)|.
\end{equation}
If the two planes do not both pass through zero, we take the associated linear subspaces.
\begin{lem} [Lemma 7.3 + 7.4 \cite{davidtoro}] \label{lem74}
For $k \geq 1$, $y \in \Sigma_k \cap V_k^8$, choose $i \in J_k$ such that $|y-x_{i,k}|\leq 10 r_k$, and let $u \in T_y\Sigma_k, |u|=1$. Then for all $j \in J_k$ such that $y \in 10B_{j,k}$, 
\begin{equation} \label{lemma73}
    \left| D\pi_{i,k} \cdot [\pi_{j,k}(y)-y] \right| \leq C\e_k(y)^2r_k,
\end{equation}
\begin{equation} \label{lemma73angle}
    \operatorname{Angle}(T\Sigma_k(y),P_{i,k}) \leq C\e_k(y),
\end{equation}
\begin{equation} \label{74proof}
    \left| D\pi_{i,k} \circ \left[D\pi_{j,k}- D\pi_{i,k}\right] \circ D\pi_{i,k} \right| \leq C \e_k(y)^2,
\end{equation}
and for every unit vector $v \in T\Sigma_k(y)$,
\begin{equation} \label{e2sigma}
    \left||D\sigma_k(y)\cdot v|-1\right| \leq  C \e_k(y)^2.
\end{equation}
\end{lem}

\begin{rk}
Equation (\ref{74proof}) is in fact (7.31) in the proof of Lemma 7.4 in \cite{davidtoro}.
\end{rk}
We now want to obtain similar estimates on the second derivatives of the $\sigma_k$.
\begin{lem} \label{lem1}
For $k \geq 0$, $y \in \Sigma_k \cap V_k^8$,  we have
\begin{equation}
\left|D^2 \sigma_k(y)\right| \leq C \frac{\e_k(y)}{r_k}.
\end{equation}
\end{lem}

\begin{proof}
Choose $i \in J_k$ such that $|y-x_{i,k}| \leq 10r_k$. Then
\begin{equation} 
D^2 \sigma_k(y)(y) =  2\sum_{j \in J_k} D\theta_{j,k}(y)\left[D\pi_{j,k}- D\pi_{i,k}\right] + \sum_{j \in J_k} [\pi_{j,k}(y)-\pi_{i,k}(y)]D^2\theta_{j,k}(y)
\end{equation} 
by (\ref{partition}). Now, when $\theta_{j,k}(y)\neq 0$,
\begin{equation}
d_{x_{i,k},100r_k}(P_{i,k},P_{j,k}) \leq \e_k(y)r_k,
\end{equation}
because $ y \in 10B_{j,k}\cap 10B_{i,k}$. Hence $|\pi_{i,k}(y)-\pi_{j,k}(y)| \leq C \e_k(y)r_k$
and $\left|D\pi_{j,k}- D\pi_{i,k}\right| \leq C \e_k(y)$. Moreover $\left|D\theta_{j,k}(y)\right| \leq C\frac{1}{r_k}$ and $\left| D^2\theta_{j,k}(y)\right| \leq C\frac{1}{r_k^2}$, so that
\begin{equation}
\left| D^2 \sigma_k(y) \right| \leq  \ C\frac{1}{r_k}  \e_k(y) + C\frac{1}{r_k^2} \e_k(y)r_k \leq C \frac{\e_k(y)}{r_k}. \qedhere
\end{equation}
\end{proof}

Recall now that by Lemma \ref{tgderivatives}, $D\sigma_k$ is bijective.  Following the same steps as above we can improve the estimates on the inverses of the $\sigma_k$'s and obtain the following lemma.

\begin{lem}
Let $v$ be a unit vector in $T\Sigma_{k+1}(z)$, and $z \in \Sigma_{k+1}\cap V_{k+1}^8$. Then 
\begin{equation} \label{zeroinverse}
\left| D\sigma^{-1}_k(y)\cdot v - v \right| \leq  C \e_k(z)|v|,
\end{equation}
\begin{equation} \label{oneinverse}
\left||D\sigma_k^{-1}(z)\cdot v|-1\right| \leq  C \e_k(z)^2,
\end{equation}
and
\begin{equation} \label{twoinverse}
\left| D^2\sigma_k^{-1}(z)\right| \leq C  \frac{\e_k(z)}{r_k}.
\end{equation}
\end{lem}

\subsection{Proof of Theorems \ref{epsilonk11} and \ref{logtheorem}} \label{proof}

Before proving Theorem \ref{epsilonk11} we need one more lemma.

\begin{lem} \label{stein}
Suppose $g_j$ is a sequence of continuous functions on $B(0,1)$, that satisfy
\begin{equation}
|g_j(x)-g_j(y)| \leq A^j |x-y| \quad \text{ for some $A>1$,}
\end{equation}
and
\begin{equation}
|g_k(x)-g_{k+1}(x)| \leq a_k(x) \ \text{ for $\{a_k(x)\}$ s.t. $\sum_{k= j}^{\infty}a_k(x) \leq C B^{-j}$, for some $B>1$}.
\end{equation}
Then the limit $g(x)=\lim_{j \to \infty}g_j(x)$ is $\eta$-H\"{o}lder continuous with H\"{o}lder seminorm $C$, where $\eta=\frac{\log B}{\log(AB)}$.
\end{lem}
The lemma is Lemma 2.8, Chapter 7 in \cite{steinshakarchi}. For convenience of the reader, we report the proof below.

\begin{proof}
First note that $g(x)$ is the limit of the uniformly convergent series
\begin{equation}
g(x)=g_1(x) + \sum_{k=1}^{\infty} (g_{k+1}(x)-g_k(x)).
\end{equation}
Then
\begin{equation}
|g(x)-g_j(x)| \leq \sum_{k=j}^{\infty} |g_{k+1}(x)-g_k(x)| \leq \sum_{k=j}^{\infty} a_k(x) \leq C B^{-j}.
\end{equation}
By the triangle inequality we get
\begin{equation}
|g(x)-g(y)| \leq |g(x)-g_j(x)| + |g_j(x)-g_j(y)|+ |g(y) -g_j(y)| \leq C(A^j |x-y| + B^{-j}).
\end{equation}
Now, for fixed $x \neq y$ we want to choose $j$ so that the two terms on the right hand side are comparable.
We want to choose $j$ such that 
\begin{equation}
(AB)^j|x-y|\leq 1 \quad \text{and} \quad 1 \leq (AB)^{j+1}|x-y|.
\end{equation}
Let $j=-\lfloor \log_{AB}|x-y| \rfloor$. Then the two inequalitites are clearly satisfied. The first one gives $A^j |x-y| \leq B^{-j}$ and by raising the second one to the power $\eta$, recalling that $(AB)^{\eta}=B$ by definition, we get that $B^{-j}\leq  |x-y|^{\eta}$. This gives
\begin{equation}
|g(x)-g(y)| \leq C(A^j |x-y| + B^{-j}) \leq C B^{-j} \leq C |x-y|^{\eta},
\end{equation}
which is what we wanted to prove.
\end{proof}


\begin{proof}[Proof of Theorem \ref{epsilonk11}]

Recall $\Sigma_0$ is a $d$-plane, so for $x,y \in \Sigma_0 \cap B(0,1)$ we can connect them through the curve $\gamma(t)=tx+(1-t)y$ on $I=[0,1]$.
We have that 
\begin{equation}
Df_m(y)-Df_m(x)= \int_I D^2f_m ( \gamma (t))\cdot \gamma'(t)\,dt.
\end{equation}

Now, set $A_k=D^2f_k ( \gamma (t))\cdot \gamma'(t)$ (note that $A_0=0$), and let $z_k=f_k(\gamma(t))$.
By the definition of the $f_k$'s we have
\begin{equation} \label{am}
A_{k+1}=D^2f_{k+1} ( \gamma (t))\cdot \gamma'(t)= D^2\sigma_k(z_k) \cdot Df_k(\gamma(t)) \cdot Df_k(\gamma(t)) \cdot \gamma'(t) + D\sigma_k(z_k) \cdot A_k.
\end{equation}

We want to estimate $A_m$. In the proof of Proposition 8.1 in \cite{davidtoro}, equation (8.10) says
\begin{equation} \label{8.10}
|Df_m(\gamma(t)) \cdot \gamma'(t)| \leq C |\gamma'(t)| \prod_{0 \leq k < m} [1+C \e_k(z_k)^2]|.
\end{equation}
If $0<x<1$ clearly $(1+x)^2\leq 1 + 3x$, so we have, by (\ref{e2sigma}), (\ref{8.10}), and Lemma \ref{lem1},
\begin{align*} \numberthis \label{A_mfirst}
|A_m| & \leq |D^2\sigma_m(z_m) \cdot Df_m(\gamma(t)) \cdot Df_m(\gamma(t)) \cdot \gamma(t)| + |D\sigma_m(z_m) \cdot A_{m-1}| \leq \\
& \leq C  \frac{\e_m(z_m)}{r_m} \prod_{0 \leq k < m} [1+C \e_k(z_k)^2]|\gamma'(t)|+ (1 + C  \e_m(z_m)^2)|A_{m-1}| = \\
& = b_m + c_m |A_{m-1}|,
\end{align*}
where we set $b_m=C\frac{\e_m(z_m)}{r_m} \prod_{0 \leq k < m} [1+C \e_k(z_k)^2]|\gamma'(t)|$ and $c_m=(1 + C  \e_m(z_m)^2)$.
We want to iterate (\ref{A_mfirst}). Recalling that $A_0=0$, 
\begin{align*} \numberthis
|A_m| & \leq b_m + c_m |A_{m-1}| \leq \\
& \leq b_m + c_m (b_{m-1} + c_{m-1}|A_{m-2}|) \leq \\
& b_m + b_{m-1}c_m + c_mc_{m-1}(b_{m-2}+c_{m-2}|A_{m-2}|) \leq \\
& \leq \dots \leq \\
& \leq \sum_{k=0}^{m} \left(b_k \prod_{j=k+1}^m c_k\right) \\
& = \sum_{k=0}^{m} \frac{\e_k(z_k)}{r_k} \prod_{i=0}^{k-1} (1+C \e_i(z_i)^2) \prod_{j=k+1}^m (1+C \e_j(z_j)^2)|\gamma'(t)|,
\end{align*}
so that,
\begin{equation} 
|A_m| \leq C \sum_{k=0}^m\left( \prod_{\Atop{0 \leq i \leq m}{i \neq k}}  [1+C \e_i(z_i)^2]\right)  \frac{\e_k(z_k)}{r_k}|\gamma'(t)|.
\end{equation}

Notice that if $\sum_{k=0}^{\infty} \frac{\e_k(f_k(z))^2}{r_k^{\alpha}}$ is finite then surely $\sum_{k=0}^{\infty} \e_k(f_k(z))^2$ also is, so Theorem \ref{ftdt} holds and in particular $\prod_{\Atop{0 \leq i \leq m}{i \neq k}}  [1+C \e_i(z_i)^2] \leq C(M)$ so
\begin{equation} \label{ambound}
|A_m| \leq C \sum_{k=0}^m \frac{ \e_k(z_k)}{r_k}|\gamma'(t)|.
\end{equation}

Then,
\begin{align*} \numberthis \label{lipbound}
\left|Df_m(y)-Df_m(x)\right| &  \leq \int_I \left|D^2f_m ( \gamma (t))|\gamma'(t)|\right|\,dt =\\
 & = \int_I |A_m|\,dt \leq \\
 & \leq C\sum_{k=0}^m  \frac{\e_k(z_k)}{r_k} \int_I |\gamma'(t)|\,dt= \\
 & = C\sum_{k=0}^m  \frac{\e_k(z_k)}{r_k}|x-y|.
\end{align*}

We now want to use Lemma \ref{stein}. By Cauchy-Schwarz,
\begin{align*} \numberthis
 \sum_{k=0}^m  \frac{\e_k(z_k)}{r_k} & = \sum_{k=0}^m  \frac{\e_k(z_k)}{r_k^{\alpha}} r_k^{\alpha-1}\leq \\
 & \leq  \left(\sum_{k=0}^m  \frac{\e_k(z_k)^2}{r_k^{2\alpha}} \sum_{k=0}^m r_k^{2\alpha-2}\right)^{\frac12} \leq \\
 & \leq C(M) \left( \sum_{k=0}^m r_k^{2\alpha-2}\right)^{\frac12} \leq \\
 & \leq C(M)r_m^{\alpha-1} = \\
 & = C(M)(10^{1-\alpha})^m.
\end{align*}
Notice that in the last inequality we used the fact that $\alpha<1$.
Let $u \in \R^n$ be a unit vector. By (\ref{lipbound}) we have
\begin{equation}
\left|Df_m(y)\cdot u -Df_m(x)\cdot u \right| \leq C(M)(10^{1-\alpha})^m|x-y|.
\end{equation}
Moreover we have, by (\ref{inverse2}), because $v=Df_m(x)\cdot u \in T\Sigma_m(y)$,
\begin{align*} \numberthis
\left|Df_{m+1}(x)\cdot u -Df_m(x)\cdot u\right| & =\left| D\sigma_m(f_m(x))Df_m(x)\cdot u-Df_m(x)\cdot u \right| \leq \\
& \leq C \e_m(x_m) \left| Df_m(x)\cdot u \right| \leq C(M)  \e_m(x_m).
\end{align*}
Then we can apply Lemma \ref{stein}, with $g_j=Df_j(x)\cdot v$, $a_k(x)=  \e_k(x_k)$, $A=10^{1-\alpha}$, and $B=10^{\alpha}$, since we know, by (\ref{beta2r_alpha}), that
\begin{align*} \numberthis
\sum_{k\geq j}  \e_k(x_k) & = \sum_{k\geq j}  \frac{ \e_k(x_k)}{r_k^{\alpha}}r_k^{\alpha} \leq \\
& \leq \left( \sum_{k\geq j}  \frac{ \e_k(x_k)^2}{r_k^{2\alpha}}\sum_{k\geq j}r_k^{2\alpha}\right)^{\frac12} \leq \\
& \leq C(M) r_j^{\alpha}
\end{align*}
Then $\eta=\frac{\log 10^{\alpha}}{\log(10)}=\alpha$ and the lemma hence gives that $Df\cdot u$ is $\alpha$-H\"{o}lder for every $u \in \R^n$.

Now, we want to prove that, for every $w \in T\Sigma(\overline{x})$, $Df^{-1}(\overline{x})\cdot w$ is $\alpha$-H\"older.


Let $x_m,y_m \in \Sigma_m$, where $m$ is such that $r_{m+1}\leq |\overline{x}-\overline{y}| \leq r_m$,  let $x_m = f_m\circ f^{-1}(\overline{x})$ and $y_m = f_m\circ f^{-1}(\overline{y})$. 
By the results in \cite{davidtoro} we know that both $f_m$ and $f^{-1}$ are bi-Lipschitz maps, so we have that $\frac{1}{C} |\overline{x}-\overline{y}| \leq |x_m-y_m| \leq C |\overline{x}-\overline{y}|$.

We want to show that, for every $m \geq 0$ we have
\begin{equation}
\left|Df_m^{-1}(y_m)-Df_m^{-1}(x_m)\right| \leq  C\sum_{k=0}^m  \frac{\e_k(z_k)}{r_k}|\overline{x}-\overline{y}|
\end{equation}

We may assume $m \geq 1$ as the result is obvious for $m=0$, given $f_0(x)=x$. 
Then we can proceed exactly as in the first part of the proof.
Now, observe that each $\sigma_k \colon \Sigma_k \to \Sigma_{k+1}$ is a $C^2$ diffeomorphism by Lemma \ref{tgderivatives}, so we can define $\sigma_k^{-1} \colon \Sigma_{k+1} \to \Sigma_k$ and $f_m^{-1} \colon \Sigma_m \to \Sigma_0$. 


Recall that by Proposition \ref{chap5dt}, we know that $\Sigma_m$ coincides with a small Lipschitz graph in $B(x_{j,m},49r_m)$. Then there is a $C^2$ curve $\gamma \colon I \to \Sigma_m$ that goes from $x_m$ to $y_m$ with length bounded above by $(1 + C \e) |x_m-y_m|\leq C |\overline{x}-\overline{y}|$.

Write
\begin{equation}
Df^{-1}_m(y_m)-Df^{-1}_m(x_m)= \int_I D^2f_m^{-1} ( \gamma (t))\cdot \gamma'(t)\,dt.
\end{equation}


By the the estimates (\ref{oneinverse}) and (\ref{twoinverse}), together with (8.22) in \cite{davidtoro}, which says
\begin{equation} \label{8.22}
|Df_m^{-1}(\gamma(t)) \cdot \gamma'(t)| \leq C |\gamma'(t)| \prod_{0 \leq k < m} [1+C \e_k(z_k)^2].
\end{equation}
we can estimate $D^2f_{m}^{-1}$ as in (\ref{am})-(\ref{ambound}), to get
\begin{equation} 
\left|D^2f_m^{-1} ( \gamma (t))\cdot \gamma'(t)\right| \leq C \sum_{k=0}^m \frac{\e_k(z_k)}{r_k} |\gamma'(t)|,
\end{equation}
where $z_k=f_k \circ f_m^{-1}(\gamma(t))$ and so
\begin{align*} \numberthis \label{inverse_bound}
\left|Df_m^{-1}(y_m)-Df_m^{-1}(x_m)\right|   & \leq \int_I \left|D^2f_m^{-1} ( \gamma (t))\right||\gamma'(t)|\,dt \leq \\
& \leq C\sum_{k=0}^m  \frac{\e_k(z_k)}{r_k}|x_m-y_m| \leq \\
& \leq C\sum_{k=0}^m  \frac{\e_k(z_k)}{r_k}|\overline{x}-\overline{y}|.
\end{align*}

Let $w \in T\Sigma(\overline{x})$. We want to apply Lemma \ref{stein} to the sequence $g_k(\overline{x})=Df_k^{-1}(x_k)\cdot w$. We have
\begin{align*} \numberthis
Df_{k+1}^{-1}(x_{k+1})\cdot w& =Df_k^{-1}(\sigma_k^{-1}(x_{k+1})) \cdot D\sigma_k^{-1}(x_{k+1})\cdot Df_k(f^{-1}(\overline{x})) \cdot Df^{-1}(\overline{x}) \cdot w = \\
& = Df_k^{-1}(x_k) \cdot D\sigma_k^{-1}(x_{k+1})\cdot v_k
\end{align*}
where we set $v_k =Df_k(f^{-1}(\overline{x})) \cdot Df^{-1}(\overline{x}) \cdot w \in T\Sigma_k(x_k)$ and we observed that $x_k=\sigma^{-1}(x_{k+1})$. Then 
\begin{align*} \numberthis
\left|Df_{m+1}^{-1}(x_{m+1})\cdot w -Df_m^{-1}(x)\cdot w\right|& = \left|  Df_m^{-1}(x_m) \cdot D\sigma_m^{-1}(x_{m+1})\cdot v_m - Df_m^{-1}(x_m)\cdot v_m\right| \leq \\
& \leq |Df_m^{-1}(x_m)| |D\sigma_m^{-1}(x_{m+1})\cdot v_m -v_m| \leq \\
& \leq C(M) |D\sigma_m^{-1}(x_{m+1})\cdot v_m -v_m| \leq \\
& \leq C(M) \e_m(x_m).
\end{align*}
where we used (\ref{8.22}) and (\ref{zeroinverse}).
Then we can apply Lemma \ref{stein} exactly as before, with $a_k(\overline{x})=  \e_k(x_k)$, $A=10^{1-\alpha}$, and $B=10^{\alpha}$, and obtain
\begin{equation}
\left|Df^{-1}(y')-Df^{-1}(x') \right|\leq C(M)|x'-y'|^{\alpha},
\end{equation}
 where $C(M)$ is a constant that depends on $M$ but not on $m$.
\end{proof}

\begin{proof} [Proof of Theorem \ref{logtheorem}]
First observe that, if we prove 
\begin{equation}
\left| Df_m(x)-Df_m(y)\right| \leq C(M) |x-y|
\end{equation}
uniformly in $m$ then the theorem follows immediately for $Df$.

Recall that, by definition, we have that
\begin{equation} \label{dini}
\sum_{k=1}^{\infty}\eta(r_k)^2 < \infty.
\end{equation}
In the same way as in the proof of Theorem \ref{epsilonk11}, we get to (\ref{lipbound}), which is
\begin{equation}
\left| Df_m(x)-Df_m(y)\right| \leq C \sum_{k=0}^m \frac{\e_k(z_k)}{r_k} |x-y|.
\end{equation}
By Cauchy-Schwarz we have
\begin{equation}
\sum_{k=0}^m \frac{\e_k(z_k)}{r_k} = \sum_{k=0}^m \frac{\e_k(z_k)}{r_k\eta(r_k)}\cdot \eta(r_k) \leq C \left( \sum_{k=0}^m \left(\frac{\e_k(z_k)}{r_k\eta(r_k)}\right)^2\sum_{k=0}^m\eta(r_k)^2\right)^\frac12 \leq C(M) \cdot C,
\end{equation}
by (\ref{dini}) and by (\ref{beta2r}).

This concludes the proof for $Df$. The same computation, combined with (\ref{inverse_bound}) from the proof of Theorem \ref{epsilonk11}, shows that $Df^{-1}$ is Lipschitz.
\end{proof}

\subsection{Proofs of Theorems \ref{theorema} and \ref{theoremb} on \texorpdfstring{$C^{1,\alpha}$}{C1alpha} parametrization} \label{improv}

We now relate the coefficients $\e_k(y)$ and the $\beta$-numbers in order to prove Theorems \ref{theorema} and \ref{theoremb}.


Let us define, as in Chapter 12 of \cite{davidtoro}, new coefficients $\gamma_k(x)$ as follows
\begin{equation}
\gamma_k(x)=d_{x,r_k}(P_{k+1}(x),P_k(x)) + \sup_{y \in E \cap B(x,35r_k)}d_{x,r_k}(P_k(x), P_k(y)).
\end{equation}
Then define, for $x \in E$, 
\begin{equation}
\hat{J}_{\gamma,\alpha}(x)=\sum_{k=0}^{\infty}\frac{\gamma_k(x)^2}{r_k^{2\alpha}}.
\end{equation}
To prove Theorem \ref{dtbinfty} in \cite{davidtoro}, the following lemma is needed. 
\begin{prop} [Corollary 12.5, \cite{davidtoro}]
If in addition to the hypotheses of Theorem \ref{f} we have that 
\begin{equation}
\hat{J}_{\gamma,0}(x) \leq M, \quad \text{ for all $x \in E$},
\end{equation}
then the map $f \colon \Sigma_0 \to \Sigma$ constructed in Theorem \ref{f} is bi-Lipschitz. Moreover the Lipschitz constants depend only on $n$, $d$, and $M$.
\end{prop}
Following the proof of Corollary 12.5 in \cite{davidtoro}, it is easy to check that under the assumption that $\hat{J}_{\gamma,\alpha}$ is uniformly bounded, the sufficient conditions in Theorem \ref{epsilonk11} are satisfied.
More specifically, we have (see page 71 of \cite{davidtoro}),
\begin{lem}
Let $z \in \Sigma_0$ and let $x \in E$ such that
\begin{equation}
    |x-f(z)| \leq 2 \dist(f(z),E).
\end{equation}
Then 
\begin{equation}
    \e_k(f_k(z)) \leq C (\gamma_k(x)+\gamma_{k-1}(x)).
\end{equation}
\end{lem}
Using the lemma, the following result follows immediately.
\begin{prop}
If in addition to the hypotheses of Theorem \ref{f} we have that 
\begin{equation}
\hat{J}_{\gamma,\alpha}(x) \leq M, \quad \text{ for all $x \in E$},
\end{equation}
then the map $f \colon \Sigma_0 \to \Sigma$ constructed in Theorem \ref{f} is invertible and differentiable, and both $f$ and its inverse have $ \alpha$-H\"{o}lder directional derivatives. In particular, $f$ is continuously differentiable. Moreover the H\"{o}lder constants depend only on $n$, $d$, and $M$.
\end{prop}

We want to replace $\hat{J}_{\gamma,\alpha}$ with a more explicit Bishop-Jones type function involving $\beta_{\infty}$'s. Define
\begin{equation}
J^E_{\alpha, \infty}(x)=\sum_{k=0}^{\infty}\frac{\beta^E_{\infty}(x,r_k)^2}{r_k^{2\alpha}}.
\end{equation}

The proof of Corollary 12.6 in \cite{davidtoro}, which we restated as Theorem \ref{dtbinfty}, can be used directly to prove Theorem \ref{theorema}, which is obtained as corollary of Theorem \ref{epsilonk11} and Theorem \ref{logtheorem}.

We would now like to replace $J^E_{\alpha, \infty}$ with $J^E_{\alpha,1}$ based on an $L^1$ version of the $\beta$-numbers. Usually such coefficients are used when the Hausdorff measure restricted to the set $E$ is Ahlfors regular. We will not need to assume such regularity, after observing that Reifenberg flatness implies lower regularity. 

\begin{lem} [Lemma 13.2, \cite{davidtoro}]
Let $E \subseteq B(0,1)$ be a Reifenberg flat set, for $\e$ small enough. Then, for $x \in E$ and for small $r>0$,
\begin{equation}
\mathcal{H}^d(\overline{E}\cap B(x,r)) \geq (1-C\e)\omega_d r^d,
\end{equation}
where $\omega_d$ denotes the measure of the unit ball in $\R^d$.
\end{lem}
\begin{rk}
We denote by $\overline{E}$ the closure of $E$, and notice that the Reifenberg flatness assumption implies that the set has no holes, for $\e$ small enough (otherwise the result would be clearly false). 
\end{rk}

The following lemma is implied by the proof of Corollary 13.1 in \cite{davidtoro}.
\begin{lem}
By changing the net $x_{j,k}$ if necessary, we have that $ \e_k
(x_k) \leq \beta^E_1(\overline{z}, r_{k-3})$, where $\overline{z} \in E$ is chosen appropriately.
\end{lem}

Using the lemma, Theorem \ref{theoremb} follows immediately from Theorem \ref{epsilonk11} and Theorem \ref{logtheorem}.

\section{Remarks and complements} \label{stuff}

\subsection{A $C^{1,\alpha}$ function which is not $C^{1,\alpha + \e}$}

As mentioned in the introduction, we now include some results from Anzellotti and Serapioni, \cite{anzellottiserapioni}. 
\begin{prop} [G. Anzellotti, R. Serapioni,  Proposition 3.2  \cite{anzellottiserapioni}] \label{c1=c2}
A $C^{k-1,1}$ $d$-rectifiable set is $C^k$ $d$-rectifiable.
\end{prop}

\begin{prop} [G. Anzellotti, R. Serapioni,  Proposition 3.3 and Appendix  \cite{anzellottiserapioni}]\label{c1snoc1t}
Let $k,m \geq 1$ and $k + s <m+t$. Then there exist $C^{k,s}$ rectifiable sets that are not $C^{m,t}$ rectifiable.
\end{prop}

For their proofs, we refer the interested reader to \cite{anzellottiserapioni}. We note that the proof of the second proposition, in the Appendix of \cite{anzellottiserapioni}, contains a small inconsequential error, which can be easily removed.

\subsection{Necessary conditions} \label{necessary}

We also record some observations in the direction of the converses of our theorems and those from \cite{davidtoro}.

\begin{prop}
Let $G$ be a Lipschitz graph in $\R^n$. Then 
\begin{equation}
    \sum_{k=0}^{\infty}\beta_{\infty}^G(x,r_k)^2 \leq M, \quad \text{ for all $x \in G$}.
\end{equation}
\end{prop}

\begin{proof}
This follows from the Main Lemma in \cite{tolsaI}, Lemma 2.1. 
\end{proof}

\begin{prop} \label{taylor}
Let $\alpha, \alpha' \in (0,1)$, $\alpha' > \alpha$ and let $G$ be a $C^{1,\alpha'}$ graph in $\R^n$. Then there exists $M >0$ such that
\begin{equation}
    J_{\infty,\alpha}^G(x)=\sum_{k=0}^{\infty}\frac{\beta_{\infty}^G(x,r_k)^2}{r_k^{\alpha}} \leq M, \quad \text{ for all $x \in G$}.
\end{equation}
\end{prop}
\begin{proof}
The proof follows the steps from Example 3.1 in \cite{naber}. Let $M$ be the graph of a $C^{1,\alpha'}$ function $f\colon \R^d \to \R^{n-d}$. By the Taylor expansion around $(x_0,f(x_0))$ we get
\begin{equation}
    |f(x)-f(x_0) - \nabla f(x_0)\cdot (x-x_0)| \leq C |x-x_0|^{1+ \alpha'}.
\end{equation}
Because $M$ is smooth we can choose the tangent plane at $x_0$ as best approximating plane in $\beta_{\infty}^G(x_0,r)$, for $r$ sufficiently small.
Then we get
\begin{equation}
    \beta_{\infty}^G(x_0,r)^2 \leq C r^{2\alpha'}.
\end{equation}
This clearly implies that
\begin{equation}
    \sum_{k=0}^{\infty}\frac{\beta_{\infty}^G(x,r_k)^2}{r_k^{2\alpha}} \leq M
\end{equation}
as $\alpha'-\alpha>0$.
\end{proof}

In Remark \ref{e:pointwisebound}, we observed that the condition $\beta_{\infty}(x,r) \leq C r^{\alpha}$ would also work for our main theorems, and as the proof above shows, such condition is satisfied by a $C^{1,\alpha}$ map.

\subsection{Sharpness of the result}
The theorems are sharp in the following sense. Let $s \in (0,1)$ and $\e \in (0,1-s)$. Let $f \in C^{1, s + \frac{\e}{2}}$ such that $f$ is purely $C^{1, s + \e}$ unrectifiable (such a function exists by Proposition \ref{c1snoc1t}). Then by Proposition \ref{taylor} we know that for the graph of $f$, $G$ we have $J_{\infty,s}^G(x)< \infty$. That is that for every $\e \in (0,1-s)$ we have a function $f$ which is purely $C^{1, s + \e}$ unrectifiable and such that $J_{\infty,s}^G(x)< \infty$.
This is the same conclusion as the second part of Theorem 1.1 in \cite{kola3}.

\subsection{How to produce H\"{o}lder functions}
We outline another more flexible construction of a $C^{1,\alpha}$ function. For a more extensive discussion on how to generate H\"older functions, see B.6 in Appendix B by S. Semmes in \cite{gromov}. We include this example as it is of different nature than the one discussed in Proposition \ref{c1snoc1t}, and we can easily estimate its Jones function.

For the remainder of this section, let $\Delta_m$ denote the collection of dyadic intervals of size $2^{-m}$, and let $\Delta =\bigcup_{m=0}^{\infty}\Delta_m$.

For $J \in \Delta$, let $h_J$ be the Haar wavelet, normalized so that $\int_J|h_J(x)|\,dx=1$ and $\int_J h_J(x)\,dx=0$, that is
\begin{equation}
h_J(x) =\begin{cases} \frac{1}{|J|} & x \in J_{l} \\
-\frac{1}{|J|} & x \in J_{r},
\end{cases}
\end{equation}
where $J_l$ and $J_r$ are the left and right half of $J$, respectively. 
Now define 
\begin{equation}
\psi_I(x) =\int_{-\infty}^x h_{I}(t)\,dt
\end{equation}
 and 
\begin{equation}
g_k(x)=\sum_{j=0}^k \sum_{J \in \Delta_j} 2^{-\alpha j} \psi_J(x),
\end{equation}
where $\alpha \in (0,1)$. By Lemma \ref{stein}, $g (x)= \lim_{k \to \infty} g_k(x)$ is a $C^{\alpha}$ function, and so
\begin{equation}
f(x)= \int_0^x g(t) \,dt
\end{equation}
is a $C^{1,\alpha}$ function.

Observe that for the function $f$ we can compute explicitly the $\beta$ numbers. Note that, because $\beta_{\infty}(x,2^{-j}) \leq C \alpha_j$ by construction, we get that the Jones function for the graph of $f$ is 
\begin{equation}
J_{\infty,\alpha'}(x) \leq C \sum_{j=1}^{\infty} \frac{\alpha_j^2}{2^{-2\alpha' j}}=C \sum_{j=1}^{\infty}2^{-2(\alpha-\alpha') j},
\end{equation}

which is in line with the discussion in Section \ref{necessary}.
\begin{figure}
  \centering
    \includegraphics[width=0.7\textwidth]{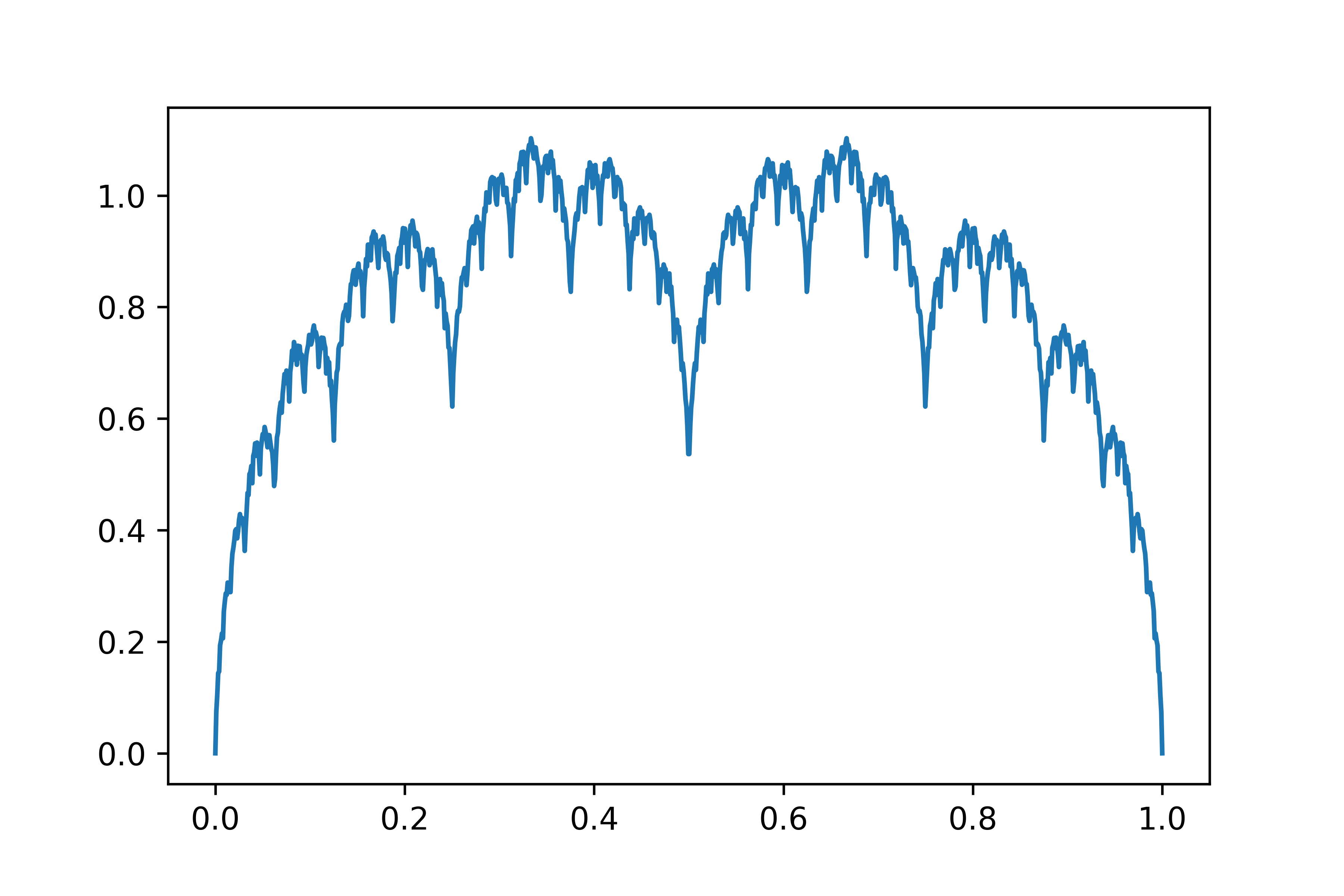}
    \caption{The function $g_k$ on $[0,1]$ for $k=10$ and $\alpha=\frac12$.}
\end{figure}

Now, we want to prove that $f$ is not $C^{1,\alpha + \e}$, for any $\e>0$. 
In order to do so, we will prove that $g$ is not $C^{\alpha+\e}$. Now, let $I$ be an interval of size $t=2^{-m}$, and let $K=2^k$ be a constant to be fixed later and write
\begin{align*} \numberthis
g(c_I)-g(x) & =\sum_{j=0}^\infty \sum_{J \in \Delta_j} 2^{-j \alpha } [\psi_J(c_I)-\psi_J(x)]= \\
& =  \sum_{j>m+k} \sum_{J \in \Delta_j} 2^{-j \alpha }[\psi_J(c_I)-\psi_J(x)] + \\
& + \sum_{m-k\leq j \leq m+k} \sum_{J \in \Delta_j}2^{-j \alpha }[\psi_J(c_I)-\psi_J(x)] + \\
& + \sum_{j < m-k} \sum_{J \in \Delta_j} 2^{-j \alpha }[\psi_J(c_I)-\psi_J(x)] = \\
& = HF + MF + LF,
\end{align*}
the high, medium and low frequencies, respectively.


Because of our normalization of the $h_J$'s, we have that $|\psi_J|\leq 1/2$. For the innermost sum, for any given $y$, at most one of the intervals $J$ of a fixed size $s$ is such that $\psi_J(y)\neq 0$. Then we have
\begin{align*} \numberthis
|HF| & \leq \sum_{j > m+k} 2^{-j \alpha }|\psi_J(c_I)| + \sum_{j > m+k} 2^{-j \alpha }|\psi_J(x)| \leq  \\
& \leq \sum_{j > m+k} 2^{-j \alpha } \leq 2^{-(m+k)\alpha+1} = \\
& =  2^{-m\alpha - k \alpha +1}.
\end{align*}

Now choose $x$ so that $|x-c_I|\leq 2^{-m-k}$.

Because of our definition of $\psi_J$, we have that $|\psi_J'(x)|=1/2^{-j}$, recalling that $J \in \Delta_j$, so that $|\psi_J(c_I)-\psi_J(x)| \leq 2^{-m}/2^{-j}=2^{-m+j}$. Moreover, because of our choice of $x$ only finitely many terms of the innermost sum are nonzero, and so we have 

\begin{align*} \numberthis
|LF| &\leq  \sum_{j < m-k} \sum_{J \in \Delta_j} 2^{-j \alpha}|\psi_J(c_I)-\psi_J(x)| \leq \\
& \leq  2\sum_{j < m-k} 2^{-j \alpha}2^{-m+j} = \\
& = 2^{-m+1} \sum_{j < m-k} 2^{j(1- \alpha)} \leq \\
& \leq 2^{-m+2} 2^{(m-k)(1-\alpha)} = \\
& = 2^{-m\alpha - k(1-\alpha) +2 } 
\end{align*}

Now, without loss of generality, we can assume $I$ and $x$ are both contained in $[0,1]$, as $g$ is periodically defined on the intervals $[n,n+1)$. Let $I=[0,2^{-n})$ and let $x$ be such that $|x| < \frac{2^{-m}}{2^{k+2}}=2^{-m-k-2}$. Then, noting that $\psi_J$ has positive slope both at $x$ and $c_I$ for our choices of $x$ and $I$, so that there is no cancellation, we get that $\psi_J(c_I)-\psi_J(x) \geq \frac14$. Finally, we get
\begin{align*} \numberthis
|MF| & = \left|\sum_{m-k\leq j \leq m+k} \sum_{J \in \Delta_j} 2^{-j\alpha}[\psi_J(c_I)-\psi_J(x)]\right| \geq \\
& \geq \frac14 \left|\sum_{m-k\leq j \leq m+k} 2^{-j\alpha}\right| \geq \\
& \geq  2^{-2 + k + 1} 2^{-(m+k)\alpha}= \\
& =2^{-m\alpha +k(1-\alpha)-1}.
\end{align*}

This means that, for infinitely many choices of $I$ and $x$, we have
\begin{align*} \numberthis
|g(c_I)-g(x)| & \geq |MF| - |HF| - |LF| \geq \\
&\geq 2^{-m\alpha +k(1-\alpha)-1}-2^{-m\alpha - k \alpha +1} - 2^{-m\alpha - k(1-\alpha) +2 } = \\
& = \left(2^{k(1-\alpha)-1}-2^{- k \alpha +1} - 2^{- k(1-\alpha) +2 }\right)2^{-m\alpha} = \\
& = \left(\frac12 K^{1-\alpha} - 2(K^{-\alpha}+K^{\alpha-1})\right)t^{\alpha},
\end{align*}
recalling that we set $K=2^k$, $t=2^{-m}$.
By choosing $K$ large enough with respect to $\alpha$, for instance by choosing $k =\frac{3}{1-\alpha}$ we get
\begin{equation}
|g(c_I)-g(x)|\geq 2 t^{\alpha}
\end{equation}
which concludes the proof.

A similar argument can be applied to many other intervals $I$. All we need is sufficiently many consecutive generations where $I$ is on the left side, to avoid cancellation. 

Thus, on one hand $g$ is a $C^{\alpha}$ function, and we just proved it is not $C^{\alpha+\e}$ for any $\e>0$, at a dense set of points, so that $f$ as above is a $C^{1,\alpha}$ function which is not $C^{1,\alpha + \e}$.

Lastly, let us mention an interesting representation for H\"older functions, which is a slight modification of the procedure presented in section B.7 in the aforementioned Appendix by S. Semmes. The idea is  similar to the one discussed in Section \ref{proof} (that is, Theorem \ref{stein}).

Let $\widetilde{\psi}_I(x) =\int_{-\infty}^x h_{3I}(t)\,dt$, where $3I$ denotes the interval with the same center as $I$ and three times its size. Moreover define a partition of unity
\begin{equation}
\phi_I(x)=\frac{\widetilde{\psi}_I(x)}{\sum_{|J|=|I|} \widetilde{\psi}_J(x)}.
\end{equation}
Clearly $0 \leq \phi_I(x) \leq 1$, it's supported on $3I$ and it is $\frac{1}{|3I|}$-Lipschitz. Moreover, for every $x \in \R$
\begin{equation}
\sum_{|I|=t} \phi_I(x) = 1
\end{equation}

\begin{figure}
  \centering
    \includegraphics[width=0.7\textwidth]{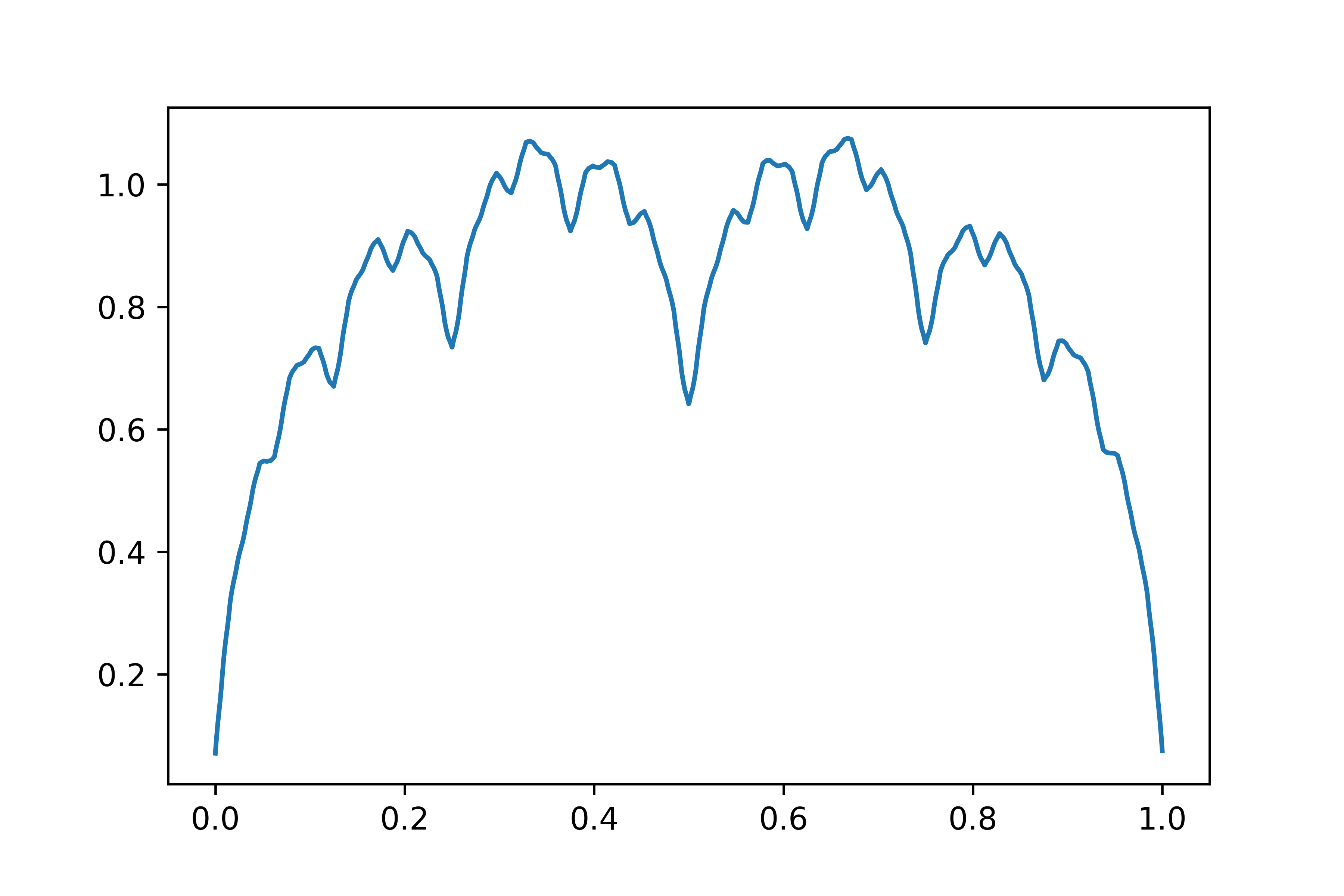}
    \caption{The function $E_t(g_k)$ on $[0,1]$ for $k=10$, $t=2^{-7}$ and $\alpha=\frac12$.}
\end{figure}

Given $G \colon \R \to \R$, a $\alpha$-H\"older function, define
\begin{equation}
E_t(G)(x)=\sum_{|I|=t} G(c_I)\phi_I(x),
\end{equation}
where the sum is over all dyadic intervals, $t=2^{-m}$, for some integer $m$, and $c_I$ denotes the center of the interval $I$. 

\begin{lem}[(Lemma B.7.8 \cite{gromov})]
There exists a constant $C$ such that
\begin{equation}
\sup_{\R} |G-E_t(G)| \leq CKt^{\alpha},
\end{equation}
 if $G$ is $\alpha$-H\"older with constant $K$.
\end{lem}

\begin{lem}[(Lemma B.7.11 \cite{gromov})]
There exists a constant $C$ such that $E_t(G)$ is $CKt^{\alpha-1}$-Lipschitz if $G$ is $\alpha$-H\"older with constant $K$.
\end{lem}

\begin{thm}
Let $G \colon \R \to \R$. Then for every $t>0$ there exists $G_t$ such that $\|G-G_t\|_{\infty} \leq Kt^{\alpha}$ and $G_t$ is $Kt^{\alpha-1}$-Lipschitz if and only if $G$ is $\alpha$-H\"older continuous with constant proportional to $K$.
\end{thm}
\begin{proof}
One direction follows directly from the lemmas above. For the other direction, let $x,y \in \R$ and set $t=|x-y|$.
\begin{align*} \numberthis
|G(x)-G(y)| & \leq | G(x)-G_t(x)| + |G_t(x)-G_t(y)| + |G_t(y)-G(y)| \leq \\
& \leq 2Kt^{\alpha} + K t^{\alpha-1}|x-y|=3Kt^{\alpha}. \qedhere
\end{align*}
\end{proof}

%

\addcontentsline{toc}{section}{References}
\setlength{\bibsep}{0pt plus 0.3ex}
\setstretch{1}
\bibliography{document}

\begin{thebibliography}{{Goe}18}

\bibitem[AS94]{anzellottiserapioni}
Gabriele Anzellotti and Raul Serapioni.
\newblock {$C^k$}-rectifiable sets.
\newblock {\em J. Reine Angew. Math.}, 453:1--20, 1994.

\bibitem[AS18]{azzamschul}
Jonas Azzam and Raanan Schul.
\newblock An analyst's traveling salesman theorem for sets of dimension larger
  than one.
\newblock {\em Mathematische Annalen}, 370(3):1389--1476, 2018.

\bibitem[AT15]{ATII}
Jonas Azzam and Xavier Tolsa.
\newblock Characterization of {$n$}-rectifiability in terms of {J}ones' square
  function: {P}art {II}.
\newblock {\em Geom. Funct. Anal.}, 25(5):1371--1412, 2015.

\bibitem[Bad19]{survey}
Matthew Badger.
\newblock Generalized rectifiability of measures and the identification
  problem.
\newblock {\em Complex Anal. Synerg.}, 5(1):Paper No. 2, 17, 2019.

\bibitem[BK12]{blatt}
Simon Blatt and S\l~awomir Kolasi\'nski.
\newblock Sharp boundedness and regularizing effects of the integral {M}enger
  curvature for submanifolds.
\newblock {\em Adv. Math.}, 230(3):839--852, 2012.

\bibitem[BNV19]{mattlisavyron}
Matthew Badger, Lisa Naples, and Vyron Vellis.
\newblock H\"{o}lder curves and parameterizations in the analyst's traveling
  salesman theorem.
\newblock {\em Adv. Math.}, 349:564--647, 2019.

\bibitem[BP17]{bishopperes}
Christopher~J. Bishop and Yuval Peres.
\newblock {\em Fractals in Probability and Analysis}, volume 162 of {\em
  Cambridge Studies in Advanced Mathematics}.
\newblock Cambridge University Press, Cambridge, 2017.

\bibitem[BS15]{raananmatt1}
Matthew Badger and Raanan Schul.
\newblock Multiscale analysis of 1-rectifiable measures: necessary conditions.
\newblock {\em Math. Ann.}, 361(3-4):1055--1072, 2015.

\bibitem[BS17]{raananmatt2}
Matthew Badger and Raanan Schul.
\newblock Multiscale analysis of 1-rectifiable measures ii: characterizations.
\newblock {\em Anal. Geom. Metr. Spaces}, 5(1):1--39, 2017.

\bibitem[BV19]{matt}
Matthew Badger and Vyron Vellis.
\newblock Geometry of measures in real dimensions via {H}\"{o}lder
  parameterizations.
\newblock {\em J. Geom. Anal.}, 29(2):1153--1192, 2019.

\bibitem[Dav91]{davidbook}
Guy David.
\newblock {\em Wavelets and singular integrals on curves and surfaces}, volume
  1465 of {\em Lecture Notes in Mathematics}.
\newblock Springer-Verlag, Berlin, 1991.

\bibitem[Del08]{delladio}
Silvano Delladio.
\newblock A sufficient condition for the {$C^2$}-rectifiability of the set of
  regular values (in the sense of {C}larke) of a {L}ipschitz map.
\newblock {\em Boll. Unione Mat. Ital. (9)}, 1(3):695--707, 2008.

\bibitem[DKT01]{DKT}
Guy David, Carlos Kenig, and Tatiana Toro.
\newblock Asymptotically optimally doubling measures and {R}eifenberg flat sets
  with vanishing constant.
\newblock {\em Comm. Pure Appl. Math.}, 54(4):385--449, 2001.

\bibitem[DNOI19]{delninobinna2019}
Giacomo Del~Nin and Kennedy Obinna~Idu.
\newblock Geometric criteria for {}$c^{1,\alpha}$ rectifiability.
\newblock {\em ArXiv e-prints}, 2019.
\newblock \url{http://arxiv.org/abs/1909.10625}.

\bibitem[Dor85a]{dorronsoro1}
Jos\'e~R. Dorronsoro.
\newblock A characterization of potential spaces.
\newblock {\em Proc. Amer. Math. Soc.}, 95(1):21--31, 1985.

\bibitem[Dor85b]{dorronsoro2}
Jos\'e~R. Dorronsoro.
\newblock Mean oscillation and {B}esov spaces.
\newblock {\em Canad. Math. Bull.}, 28(4):474--480, 1985.

\bibitem[DS91]{davidsemmes91}
G.~David and S.~Semmes.
\newblock Singular integrals and rectifiable sets in {${\R}^n$}: {B}eyond
  {L}ipschitz graphs.
\newblock {\em Ast\'erisque}, 193:152, 1991.

\bibitem[DS93]{davidsemmes93}
Guy David and Stephen Semmes.
\newblock {\em Analysis of and on uniformly rectifiable sets}, volume~38 of
  {\em Mathematical Surveys and Monographs}.
\newblock American Mathematical Society, Providence, RI, 1993.

\bibitem[DT12]{davidtoro}
Guy David and Tatiana Toro.
\newblock Reifenberg parameterizations for sets with holes.
\newblock {\em Mem. Amer. Math. Soc.}, 215(1012):vi+102, 2012.

\bibitem[ENV16]{naber}
N.~{Edelen}, A.~{Naber}, and D.~{Valtorta}.
\newblock {Quantitative Reifenberg theorem for measures}.
\newblock {\em ArXiv e-prints}, December 2016.

\bibitem[ENV19]{naber2}
Nick Edelen, Aaron Naber, and Daniele Valtorta.
\newblock Effective {R}eifenberg theorems in {H}ilbert and {B}anach spaces.
\newblock {\em Math. Ann.}, 374(3-4):1139--1218, 2019.

\bibitem[Fed69]{federer}
Herbert Federer.
\newblock {\em Geometric measure theory}.
\newblock Die Grundlehren der mathematischen Wissenschaften, Band 153.
  Springer-Verlag New York Inc., New York, 1969.

\bibitem[FIL16]{fefferman}
Charles Fefferman, Arie Israel, and Garving~K. Luli.
\newblock Finiteness principles for smooth selection.
\newblock {\em Geom. Funct. Anal.}, 26(2):422--477, 2016.

\bibitem[GG19]{ghinassigoering2019}
Silvia Ghinassi and Max Goering.
\newblock Menger curvatures and $c^{1,\alpha}$ rectifiability of measures.
\newblock {\em Archiv der Mathematik},
  https://doi.org/10.1007/s00013-019-01414-6, 2019.

\bibitem[GKS10]{GKS}
John Garnett, Rowan Killip, and Raanan Schul.
\newblock A doubling measure on {$\mathbb R^d$} can charge a rectifiable curve.
\newblock {\em Proc. Amer. Math. Soc.}, 138(5):1673--1679, 2010.

\bibitem[{Goe}18]{othermax}
M.~{Goering}.
\newblock {Characterizations of countably $n$-rectifiable Radon measures by
  higher-dimensional Menger curvatures}.
\newblock {\em ArXiv e-prints}, April 2018.

\bibitem[Gro99]{gromov}
Misha Gromov.
\newblock {\em Metric structures for {R}iemannian and non-{R}iemannian spaces},
  volume 152 of {\em Progress in Mathematics}.
\newblock Birkh\"{a}user Boston, Inc., Boston, MA, 1999.
\newblock Based on the 1981 French original [ MR0682063 (85e:53051)], With
  appendices by M. Katz, P. Pansu and S. Semmes, Translated from the French by
  Sean Michael Bates.

\bibitem[Jon90]{atst}
Peter~W. Jones.
\newblock Rectifiable sets and the traveling salesman problem.
\newblock {\em Invent. Math.}, 102(1):1--15, 1990.

\bibitem[Kol15]{kola2}
S\l~awomir Kolasi\'nski.
\newblock Geometric {S}obolev-like embedding using high-dimensional
  {M}enger-like curvature.
\newblock {\em Trans. Amer. Math. Soc.}, 367(2):775--811, 2015.

\bibitem[Kol17]{kola3}
S\l~awomir Kolasi\'{n}ski.
\newblock Higher order rectifiability of measures via averaged discrete
  curvatures.
\newblock {\em Rev. Mat. Iberoam.}, 33(3):861--884, 2017.

\bibitem[KS13]{polish}
S{\l}awomir Kolasi\'nski and Marta Szuma\'nska.
\newblock Minimal {H}\"older regularity implying finiteness of integral
  {M}enger curvature.
\newblock {\em Manuscripta Math.}, 141(1-2):125--147, 2013.

\bibitem[L{\'e}g99]{leger}
J.~C. L{\'e}ger.
\newblock Menger curvature and rectifiability.
\newblock {\em Ann. of Math. (2)}, 149(3):831--869, 1999.

\bibitem[LW09]{lerman2}
Gilad Lerman and J.~Tyler Whitehouse.
\newblock High-dimensional {M}enger-type curvatures. {II}. {$d$}-separation and
  a menagerie of curvatures.
\newblock {\em Constr. Approx.}, 30(3):325--360, 2009.

\bibitem[LW11]{lerman1}
Gilad Lerman and J.~Tyler Whitehouse.
\newblock High-dimensional {M}enger-type curvatures. {P}art {I}: {G}eometric
  multipoles and multiscale inequalities.
\newblock {\em Rev. Mat. Iberoam.}, 27(2):493--555, 2011.

\bibitem[Mat95]{mattila}
Pertti Mattila.
\newblock {\em Geometry of sets and measures in {E}uclidean spaces}, volume~44
  of {\em Cambridge Studies in Advanced Mathematics}.
\newblock Cambridge University Press, Cambridge, 1995.
\newblock Fractals and rectifiability.

\bibitem[Mer16]{jessica}
Jessica Merhej.
\newblock {\em On the Geometry of Rectifiable Sets with Carleson and
  Poincar{\'e}-type Conditions}.
\newblock ProQuest LLC, Ann Arbor, MI, 2016.
\newblock Thesis (Ph.D.)--University of Washington.

\bibitem[Meu18]{meurer}
Martin Meurer.
\newblock Integral {M}enger curvature and rectifiability of {$n$}-dimensional
  {B}orel sets in {E}uclidean {$N$}-space.
\newblock {\em Transactions of the American Mathematical Society},
  370(2):1185--1250, 2018.

\bibitem[MO18]{orponen}
Henri Martikainen and Tuomas Orponen.
\newblock Boundedness of the density normalised jones' square function does not
  imply 1-rectifiability.
\newblock {\em J. Math. Pures Appl.}, 110:71--92, 2018.

\bibitem[Oki92]{okikiolu}
Kate Okikiolu.
\newblock Characterization of subsets of rectifiable curves in {${ \R}^n$}.
\newblock {\em J. London Math. Soc. (2)}, 46(2):336--348, 1992.

\bibitem[Paj96a]{pajot2}
Herv\'{e} Pajot.
\newblock Sous-ensembles de courbes {A}hlfors-r\'{e}guli\`eres et nombres de
  {J}ones.
\newblock {\em Publ. Mat.}, 40(2):497--526, 1996.

\bibitem[Paj96b]{pajot1}
Herv\'e Pajot.
\newblock Un th\'eor\`eme g\'eom\'etrique du ``voyageur de commerce'' en
  dimension {$2$}.
\newblock {\em C. R. Acad. Sci. Paris S\'er. I Math.}, 323(1):13--16, 1996.

\bibitem[Pra17]{marti}
Mart\'i Prats.
\newblock Sobolev regularity of the {B}eurling transform on planar domains.
\newblock {\em Publ. Mat.}, 61:291--336, 2017.

\bibitem[PTT09]{PTT}
D.~Preiss, X.~Tolsa, and T.~Toro.
\newblock On the smoothness of {H}\"older doubling measures.
\newblock {\em Calc. Var. Partial Differential Equations}, 35(3):339--363,
  2009.

\bibitem[Rei60]{reifen}
E.~R. Reifenberg.
\newblock Solution of the {P}lateau problem for {$m$}-dimensional surfaces of
  varying topological type.
\newblock {\em Bull. Amer. Math. Soc.}, 66:312--313, 1960.

\bibitem[San19]{santilli}
Mario Santilli.
\newblock Rectifiability and approximate differentiability of higher order for
  sets.
\newblock {\em Indiana Univ. Math. J.}, 68(3):1013--1046, 2019.

\bibitem[Sch07]{raananhilbert}
Raanan Schul.
\newblock Subsets of rectifiable curves in {H}ilbert space---the analyst's
  {TSP}.
\newblock {\em J. Anal. Math.}, 103:331--375, 2007.

\bibitem[SS05]{steinshakarchi}
Elias~M. Stein and Rami Shakarchi.
\newblock {\em Real analysis}, volume~3 of {\em Princeton Lectures in
  Analysis}.
\newblock Princeton University Press, Princeton, NJ, 2005.
\newblock Measure theory, integration, and Hilbert spaces.

\bibitem[Tol14]{tolsabook}
Xavier Tolsa.
\newblock {\em Analytic capacity, the {C}auchy transform, and non-homogeneous
  {C}alder\'on-{Z}ygmund theory}, volume 307 of {\em Progress in Mathematics}.
\newblock Birkh\"auser/Springer, Cham, 2014.

\bibitem[Tol15]{tolsaI}
Xavier Tolsa.
\newblock Characterization of {$n$}-rectifiability in terms of {J}ones' square
  function: part {I}.
\newblock {\em Calc. Var. Partial Differential Equations}, 54(4):3643--3665,
  2015.

\bibitem[Tol19]{tolsaarxiv}
Xavier Tolsa.
\newblock Rectifiability of measures and the {$\beta_p$} coefficients.
\newblock {\em Publ. Mat.}, 63(2):491--519, 2019.

\bibitem[Tor95]{toro1995}
Tatiana Toro.
\newblock Geometric conditions and existence of bi-{L}ipschitz
  parameterizations.
\newblock {\em Duke Math. J.}, 77(1):193--227, 01 1995.

\bibitem[Vil19]{michele}
Michele Villa.
\newblock Tangent points of lower content $d$-regular sets and $\beta$-numbers.
\newblock {\em Journal of the London Mathematical Society},
  https://doi.org/10.1112/jlms.12275, 2019.

\end{thebibliography}
\bibliographystyle{alpha}
\vspace{2cm}
\textsc{School of Mathematics, Institute for Advanced Study, Princeton, NJ 08540, USA}

\emph{E-mail address}: \texttt{ghinassi@math.ias.edu}

\end{document}